\documentclass[letterpaper,fleqn]{article}

\usepackage[margin=1.25in]{geometry}
\usepackage{marginnote}

\usepackage{mathptmx}
\usepackage[T1]{fontenc}
\usepackage{euscript}

\usepackage{latexsym,epsfig,verbatim}
\usepackage{amsmath,amsthm,amssymb}
\usepackage{colonequals}
\usepackage{units}
\usepackage{leftidx}

\usepackage{url}
\usepackage{color}
\usepackage[colorlinks,citecolor=blue,linkcolor=red!80!black]{hyperref}
\usepackage[nameinlink]{cleveref}
\usepackage{comment}

\usepackage{accents}

\usepackage{tikz}
\usetikzlibrary{matrix,arrows,decorations.pathmorphing}
\usepackage{pinlabel}

\theoremstyle{plain}
\newtheorem{theorem}{Theorem}[section]
\newtheorem{lemma}[theorem]{Lemma}
\newtheorem{corollary}[theorem]{Corollary}

\newtheorem{proposition}[theorem]{Proposition}

\newtheorem{fact}[theorem]{Fact}
\crefname{fact}{Fact}{Facts}
\newtheorem{claim}[theorem]{Claim}
\crefname{claim}{Claim}{Claims}
\newtheorem*{claim*}{Claim}
\newtheorem*{main:qi-into-CS}{\Cref{th:qi_into}}
\newtheorem*{main:cs_boundary}{\Cref{th:boundary_of_cs}}
\newtheorem*{main:finite_index_lift}{\Cref{prop:fi_lift}}
\newtheorem*{main:scott_swarup}{\Cref{th:scott_swarup}}
\newtheorem*{main:convex_cocompact}{\Cref{th:cc}}

\theoremstyle{definition}
\newtheorem{definition}[theorem]{Definition}
\newtheorem{remark}[theorem]{Remark}
\newtheorem{question}[theorem]{Question}
\crefname{question}{Question}{Questions}

\newcommand{\define}[1]{\textbf{#1}}

\newcommand{\R}{\mathbb{R}}
\newcommand{\Z}{\mathbb{Z}}

\newcommand{\N}{\mathbb{N}}
\DeclareMathOperator{\diam}{diam}
\newcommand{\floor}[1]{\left\lfloor#1\right\rfloor}
\newcommand{\ceil}[1]{\left\lceil#1\right\rceil}
\newcommand{\norm}[1]{\left\Vert#1\right\Vert}
\newcommand{\abs}[1]{\left\vert#1\right\vert}
\newcommand{\inv}{^{-1}}

\DeclareMathOperator{\Out}{Out}
\DeclareMathOperator{\Aut}{Aut}
\DeclareMathOperator{\Inn}{Inn}
\DeclareMathOperator{\rank}{rk}
\newcommand{\free}{\mathbb{F}} 
\newcommand{\factor}{{\EuScript F}} 
\newcommand{\F}{\factor} 
\newcommand{\fc}{\factor} 
\renewcommand{\int}{\mathcal{I}} 

\newcommand{\pl}{{\EuScript{PL}}} 
\newcommand{\cv}{cv} 
\newcommand{\cvbar}{\overline{\cv}} 
\newcommand{\os}{{\EuScript X}} 
\newcommand{\osbar}{\overline{\os}} 
\newcommand{\X}{\os} 
\newcommand{\dsym}{d^{\mathrm{sym}}_\os} 
\newcommand{\arational}{{\EuScript A}{\EuScript T}} 
\newcommand{\fproj}{\pi_\fc} 
\newcommand{\plproj}{\pi_\pl} 
\newcommand{\csproj}{\pi_\cs} 

\newcommand{\rose}{{\EuScript R}} 
\DeclareMathOperator{\Lip}{Lip} 
\newcommand{\sym}{{\sf M}} 
\newcommand{\CS}{{\EuScript{CS}}} 
\newcommand{\cs}{\CS} 
\DeclareMathOperator{\supp}{supp} 
\newcommand{\curr}{\Currents(\free)} 
\DeclareMathOperator{\Currents}{Curr}
\newcommand{\pcurr}{\mathbb{P}\curr} 
\newcommand{\db}{\partial^2} 
\newcommand{\len}{\ell} 
\newcommand{\bund}{\mathcal{E}} 
\newcommand{\base}{\mathcal{B}} 
\renewcommand{\sym}{{\sf M}} 
\newcommand{\legal}{\mathrm{leg}} 
\newcommand{\leg}{\legal} 
\newcommand{\illegality}{{\EuScript I}} 
\newcommand{\e}{\mathbf{el}}
\newcommand{\epl}{\e}
\newcommand{\efc}{\e'}

\newcommand{\fcbound}{{\sf B}} 
\newcommand{\bblip}{\mathfrak{L}} 
\newcommand{\qisec}{\mathfrak{K}} 
\newcommand{\fiberlip}{\mathfrak{C}} 
\newcommand{\separation}{\mathfrak{D}} 

\newcommand{\nbhd}[2]{{\EuScript N}_{#1}(#2)} 
\newcommand{\dhaus}{d_{\mathrm{Haus}}} 

\newcommand{\I}{\mathbf{I}}
\newcommand{\Ipl}{\I_+}
\newcommand{\Imin}{\I_-}
\newcommand{\J}{\mathbf{J}}

\newcommand{\cay}[2]{\mathrm{Cay}({#2}, {#1})} 

\newcommand{\minlen}[2]{{\sf m}_{#1}^{#2}} 
\newcommand{\minset}[2]{\rho_{#1}^{#2}} 
\newcommand{\axis}{\mathcal{A}} 
\DeclareMathOperator{\width}{width}
\newcommand{\fiberproj}{\mathfrak{p}} 
\newcommand{\fiberwise}{\fiberproj}

\begin{document}

\renewcommand{\thefootnote}{\fnsymbol{footnote}} 
\footnotetext{\emph{Key words and phrases:} Hyperbolic extensions, convex cocompact, co-surface graph, alignment-preserving maps, fiber distortion} 
\renewcommand{\thefootnote}{\arabic{footnote}}

\title{\textbf{\Large  The co-surface graph and the geometry of hyperbolic free group extensions
}}
\author{Spencer Dowdall and Samuel J. Taylor
\thanks{
The authors were each partially supported by NSF postdoctoral fellowships; the first by grant DMS-1204814 and the second by grant DMS-1400498.}}
\date{\today}

\maketitle

\begin{abstract}
We introduce the co-surface graph $\cs$ of a finitely generated free group $\free$ and use it to study the geometry of hyperbolic group extensions of $\free$. Among other things, we show that the Gromov boundary of the co-surface graph is equivariantly homeomorphic to the space of free arational $\free$--trees and use this to prove that a finitely generated subgroup of $\Out(\free)$ quasi-isometrically embeds into the co-surface graph if and only if it is purely atoroidal and quasi-isometrically embeds into the free factor complex. This answers a question of I. Kapovich. Our earlier work \cite{DT1} shows that every such group gives rise to a hyperbolic extension of $\free$, and here we prove a converse to this result that characterizes the hyperbolic extensions of $\free$ arising in this manner. As an application of our techniques, we additionally obtain a Scott--Swarup type theorem for this class of extensions.
\end{abstract}
\smallskip

\section{Introduction}
\label{sec:intro}
Let $\free$ be the free group of rank $r \ge 3$ and let $\Out(\free)$ be its outer automorphism group. Every subgroup $\Gamma \le \Out(\free)$ 
gives rise to an exact sequence
\begin{align}\label{ses}
1 \longrightarrow \free \overset{i}{\longrightarrow} E_\Gamma \overset{p}{\longrightarrow} \Gamma \longrightarrow 1,
\end{align}
in which $E_\Gamma$ is the preimage of $\Gamma$ under the homomorphism $\Aut(\free)\to\Out(\free)$ and $\free\lhd E_\Gamma$ is identified with the inner automorphisms $\Inn(\free)$. In fact every group extension of $\free$ surjects onto an extension $E_\Gamma$ of this form. In \cite{DT1}, we gave conditions on $\Gamma \le \Out(\free)$ that guarantee the associated extension $E_\Gamma$ is Gromov hyperbolic. To state these conditions, first recall that $\Gamma$ is purely atoroidal if each infinite order element is atoroidal (no power fixes a nontrivial conjugacy class of $\free$) and that the free factor complex $\fc$ is the simplicial complex whose $k$--simplices are nested chains $A_0 <\dotsb< A_k$ of proper free factors of $\free$, up to conjugacy. Outer automorphisms act isometrically on $\fc$, and we say that a finitely generated subgroup $\Gamma\le\Out(\free)$ \define{qi-embeds} into $\fc$ if some (equivalently any) orbit map $\Gamma\to \fc$ is a quasi-isometric embedding.

\begin{theorem}[\cite{DT1}] \label{th:DT1}
Suppose that a finitely generated subgroup $\Gamma \le \Out(\free)$ is purely atoroidal and qi-embeds into $\fc$. Then the free group extension $E_\Gamma$ in \Cref{ses} is hyperbolic.
\end{theorem}

The goal of the present paper is twofold: to refine these conditions for hyperbolicity, and to make a more in depth study of the geometry of these hyperbolic extensions. This study culminates in a converse to the above result that characterizes those hyperbolic extensions arising from \Cref{th:DT1}. We note that the exact converse of \Cref{th:DT1} is well-known to be false. 

Following Hamenst\"adt and Hensel \cite{HaHe}, a subgroup $\Gamma\le\Out(\free)$ is said to be \define{convex cocompact} if it qi-embeds into $\fc$. However,  unlike the analogous situation for mapping class groups \cite{KentLein,H}, convex cocompactness itself does not ensure hyperbolicity of $E_\Gamma$. Indeed, pure atoroidality of $\Gamma$ is essential for $E_\Gamma$ to be hyperbolic, since a periodic conjugacy class for $\phi\in \Gamma$ gives rise to a $\Z\oplus \Z$ in $E_\Gamma$. Further, there are automorphisms of $\free$ that act loxodromically on $\fc$ but are not atoroidal. In fact, by combining work of Bestvina--Handel \cite{BH92} and Bestvina--Feighn \cite{BFhyp}, such automorphisms precisely correspond to pseudo-Anosov homeomorphisms of once-punctured, possibly nonorientable, surfaces.

This suggests that $\fc$ is not the correct complex for studying hyperbolic extensions of $\free$. It is natural to build a better-suited complex by starting with $\fc$ and coning off the curve graphs for all once-punctured surfaces $S$ with $\pi_1(S)\cong \free$. Versions of this construction appear several places in the literature---first in the work of Kapovich--Lustig \cite{kapovich2007geometric} and later Mann--Reynolds \cite{MR2} and Mann \cite{Mann-thesis} (see \S\ref{sec:co-surface graph})---in each case producing a hyperbolic $\Out(\free)$--graph $Y$ with the property that any subgroup $\Gamma$ that qi-embeds into $Y$ is both convex cocompact and purely atoroidal. The converse was posed as a question by I. Kapovich:

\begin{question}[I. Kapovich] \label{q:Kap}
Suppose that $\Gamma \le \Out(\free)$ is purely atoroidal and convex compact. Is the orbit map $\Gamma \to Y$ a quasi-isomeric embedding?
\end{question}

To answer \Cref{q:Kap}, we introduce (\S\ref{sec:co-surface graph}) a new model for the graph $Y$ that is both simple to define and natural for our purposes. 
This \define{co-surface graph} is defined to be the simplicial graph $\cs$ whose vertices are conjugacy classes of primitive elements of $\free$ and where two conjugacy classes are adjacent if there is a once-punctured surface $S$ with $\pi_1(S) \cong \free$ in which they are both represented by simple closed curves on $S$.

\begin{main:qi-into-CS}[Qi-embedding into $\cs$]
Let $\Gamma$ be finitely generated subgroup of $\Out(\free)$. Then $\Gamma$ qi-embeds into the co-surface graph $\cs$ if and only if $\Gamma$ is purely atoroidal and convex cocompact.
\end{main:qi-into-CS}

After formulating \Cref{q:Kap}, Kapovich showed it cannot be answered from formal properties of the action $\Gamma \curvearrowright Y$. That is, Kapovich constructs an action of the free group of rank $2$ on a hyperbolic graph $X$ which has all the properties of the action $\Gamma \curvearrowright Y$ but whose orbit map $\Gamma\to X$ is not a qi-embedding \cite{kapovich2015purely}. 
Thus the proof of \Cref{th:qi} necessarily requires a deeper understanding of the co-surface graph itself. Indeed, our argument uses the fine geometric structure of Culler and Vogtmann's Outer space $\os$ and the following calculation of the Gromov boundary of $\cs$:

\begin{main:cs_boundary}[Boundary of $\cs$]
The Gromov boundary $\partial \cs$ of the co-surface graph is $\Out(\free)$--equivariantly homeomorphic to the subspace of $\partial\fc$ consisting of classes of free arational trees. 
\end{main:cs_boundary}

We obtain \Cref{th:boundary_of_cs} as a corollary of the general theory of \define{alignment-preserving maps} that we develop in \S\ref{sec:alignmentpres} and which may be of independent interest. Briefly, three (ordered) points are coarsely aligned if the triangle inequality for them is nearly an equality, and a map that respects this condition is said to be alignment preserving. We show (\Cref{th:boundaries}) that any coarsely surjective alignment preserving map $X\to Y$ between hyperbolic metric spaces extends to a homeomorphism between $\partial Y$ and a specific subset of $\partial X$.

The co-surface graph $\CS$ has other advantages over the factor complex $\fc$, and indeed this is a major theme of the present paper. For example, it is well known (see \cite{BFHlam}) that full irreducibility is not stable under passage to finite index subgroups. This causes complications when attempting to study the subgroup structure of $E_\Gamma$. However, the following result shows that this is not an issue for $\cs$:

\begin{main:finite_index_lift} 
Let $H$ be a finite index subgroup of $\free$ and let $\Gamma^H$ denote the subgroup of $\Out(H)$ induced by elements of $\Gamma\le\Out(\free)$ that stabilize the conjugacy class of $H$. If $\Gamma$ is finitely generated and qi-embeds into $\cs$, then $\Gamma^H$ also qi-embeds into $\cs(H)$.
\end{main:finite_index_lift}

\Cref{prop:fi_lift} is one of the key ingredients allowing us to establish a Scott--Swarup \cite{scott1990geometric} type theorem for these extensions of $\free$.
 Recall that the fiber subgroup $\free$, being infinite and normal, is exponentially distorted in the hyperbolic group $E_\Gamma$. The following theorem, however, shows that such distortion is confined to finite index subgroups of $\free$; this mirrors a result of Dowdall--Kent--Leininger \cite{dowdall2014pseudo} for hyperbolic surface group extensions.
Bear in mind that the statement is false without the hypothesis that $\Gamma$ qi-embed into $\cs$ (see \S\ref{sec:ScSw}).

\begin{main:scott_swarup}[Nondistortion in fibers]
Suppose that $\Gamma\le\Out(\free)$ quasi-isometrically embeds into $\cs$, and let $L$ be a finitely generated subgroup of the fiber $\free\lhd E_\Gamma$. Then $L$ is quasiconvex, and hence undistorted, in the hyperbolic extension $E_\Gamma$ if and only if $L$ has infinite index in $\free$.
\end{main:scott_swarup}

\noindent We note that Mj and Rafi \cite{MjRafi} have recently, and independently, proven \Cref{th:main_ss} by very different methods. Their approach uses structural results on convex cocompact subgroups proven in \cite{DT1} as well as a characterization of the Cannon--Thurston map for \eqref{ses} that we obtained with Kapovich in \cite{DKT} and which builds on earlier work of Mj \cite{MitraEndingLams}. 
Our proof is more direct and proceeds as follows.

The second key ingredient needed to prove \Cref{th:main_ss} 
is a careful study of the geometry of hyperbolic extensions $E_\Gamma$ that focuses on the relationship between the ``local'' axis of an element $a\in \free$ acting on any given fiber of $p\colon E_\Gamma\to\Gamma$ and the ``global'' axis for $a$ acting on $E_\Gamma$. Specifically, if $a^*$ denotes the geodesic in $E_\Gamma$ whose endpoints are the fixed points $a^{\pm\infty}$ in $\partial E_\Gamma$, we define the \define{width} of $a$ to be the quantity
\[
\width(a) = \diam_\Gamma p(a^*).
\]
This concept was first studied in the context of surface group extensions by Kent and Leininger \cite{KentLein-geometric}. We prove (\Cref{width}) that when $\Gamma$ is convex cocompact, the quantity $\width(a)$ is uniformly bounded over all simple elements $a\in \free$, where an element is \define{simple} if it is contained in some proper free factor of $\free$. As a consequence, we show that the global axis $a^*$ fellow travels the local axis for $a\in \free$ acting on the fiber of $E_\Gamma\to \Gamma$ in which the translation length of $a$ is minimized. Combining this with \Cref{prop:fi_lift} leads to \Cref{th:main_ss}. This analysis also allows us to prove the following theorem, which gives a converse to \Cref{th:DT1} and characterizes hyperbolic extensions arising from convex cocompact subgroups as those for which the simple elements have uniformly bounded width:

\begin{main:convex_cocompact}[Convex cocompactness]
Suppose that $1 \to \free \to E \to Q \to 1$ is a hyperbolic extension of $\free$. Then $Q$ has convex cocompact image in $\Out(\free)$ (and hence admits a quasi-isometric embedding orbit map into $\cs$) if and only there exists $D\ge 0$ so that $\width_Q(a)\le D$ for each simple element $a\in \free$.
\end{main:convex_cocompact}

\paragraph{Acknowledgments:} The authors thank Ilya Kapovich and Patrick Reynolds for helpful conversations. We are also grateful for the hospitality of the University of Illinois at Urbana-Champaign for hosting the second author at the start of this project. Finally, we thank the referee for a very carefull reading of the paper and several useful suggestions.

\section{Background} 
\label{sec:background}

Throughout, $\free$ will denote a finitely generated free group of rank $r =\rank(\free)$ at least $3$. 
 In this section we review several structures associated to $\free$ that will be relevant to our work.

\subsection{Coarse geometry}
\label{sec:coarse_geom}

A map $f\colon X\to Y$ of metric spaces is a \define{$K$--quasi-isometric embedding} if
\[d_X(a,b)/K - K \le d_Y(f(a),f(b))\le Kd_X(a,b)+K\]
for all $a,b\in X$. The map is moreover a \define{$K$--quasi-isometry} if its image is $K$--dense in $Y$. A \define{$K$--quasigeodesic} is then a $K$--quasi-isometric embedding of an interval $\I\subset\R$ into a metric space. The \define{Hausdorff distance} between two subsets $A,B$ of a metric space $X$ is the infimum of all $\epsilon > 0$ for which $A$ and $B$ are both contained within the $\epsilon$--neighborhoods of each other. 

A geodesic metric space $X$ is \define{$\delta$--hyperbolic}, where $\delta\ge0$, if every geodesic triangle in $X$ is $\delta$--thin, meaning that each side is contained within the $\delta$--neighborhood of the other two. Every such space has a well defined \define{Gromov boundary} $\partial X$ consisting of equivalence classes of admissible sequences in $X$, where a sequence $\{a_n\}$ is admissible if $\lim_{n,m} (a_n\vert a_m)_x=\infty$ and two sequence $\{a_n\}$ and $\{b_n\}$ are equivalent if $\lim_{n,m} (a_n\vert b_m)_x = \infty$ for some $x\in X$. Here, $(a\vert b)_x$ denotes the \define{Gromov product} $(d(a,x) + d(b,x) - d(a,b))/2$. One says the admissible sequence $a_1,a_2,\dotsc\in X$ \define{converges} to the point $\{a_n\}\in \partial X$. 
The Gromov product, with respect to $x\in X$, may be extended to $\partial X$ by declaring
\[(a\vert b)_x \colonequals \sup \liminf_{m,n\to \infty} \;(a_m\vert b_n)_x,\]
where the supremum here is over all sequences $\{a_m\},\{b_n\}$ converging respectively to $a,b\in \partial X$. The boundary $\partial X$ is then equipped with the topology in which the sets $N_a^x(r)\colonequals \{b\in \partial X : (a\vert b)_x \ge r\}$ give a basis of open neighborhoods about the point $a\in \partial X$. Moreover, 
the topologies on $X$ and $\partial X$ may be extended to a topology on the disjoint union $X\cup \partial X$ for which a sequence $x_n\in X$ converges to $\zeta\in \partial X$ if and only if $\{x_n\}$ is admissible and equivalent to $\zeta$. When $X$ is \define{proper} (meaning that closed metric balls are compact), $X\cup \partial X$ is a compactification of $X$.
Finally, hyperbolicity itself and the Gromov boundary $\partial X$ are both quasi-isometry invariants of $X$.
See \cite[Section III.H.3]{BH}, \cite{GhysdelaHarpe}, or \cite[Section 2]{KapovichBenakli-boundaries} for more details.

If $X$ is $\delta$--hyperbolic, then every quasigeodesic ray $r\colon [0,\infty)\to X$ converges to its \define{endpoint at infinity} $r(\infty) \colonequals \{r(n)\}_{n=1}^\infty\in \partial X$, and any two rays whose images have finite Hausdorff distance determine the same endpoint. Conversely, as explained in \cite[Remark 2.16]{KapovichBenakli-boundaries}, for any $x_0\in X$ and $\zeta\in \partial X$ one may build a $10\delta$ quasigeodesic $r\colon \R_+\to X$ with the properties that $r(0)=x_0$ and $r(\infty) = \zeta$.

Throughout, we will use $\I$ (and sometimes $\J$) to denote a closed subinterval of $\R$. We write $\Imin$ and $\Ipl$ for the infimum and supremum of $\I$, respectively.
With this notation, every quasigeodesic $\gamma\colon \I\to X$ naturally has two well-defined \define{endpoints} $\gamma(\I_\pm)\in X\cup \partial X$, where $\gamma(\I_\pm)\in \partial X$ if $\I_\pm = \pm\infty$ and $\gamma(\I_\pm)\in X$ otherwise.  The following is a fundamental feature of hyperbolic metric spaces; see \cite[Theorem III.H.1.7]{BH} for a proof. 

\begin{proposition}[Stability of quasigeodesics]
\label{prop:stability_of_quasis}
For every $K\ge 1$ and $\delta\ge 0$ there exists a \define{stability constant} $R_0(K,\delta)>0$ such that if $\gamma\colon \I\to X$ and $\rho\colon \J\to X$ are $K$--quasigeodesics with the same endpoints in $X \cup \partial X$ for a $\delta$--hyperbolic space $X$, then $\gamma(\I)$ and $\rho(\J)$ have Hausdorff distance at most $R_0(K,\delta)$.
\end{proposition}

\subsection{Currents and laminations}
\label{sec:currents}
All finite-valent Cayley graphs of $\free$ are quasi-isometric hyperbolic spaces, and we write $\partial\free$ to denote their common Gromov boundary. The free group $\free$ acts on each of its Cayley graphs by left multiplication, and this extends to a left action of $\free$ on $\partial\free$ by homeomorphisms. Let $\db\free = \{(\eta,\xi) \mid \eta,\xi\in \partial \free, \eta\neq\xi\}$ denote the double boundary of $\free$, equipped with the subspace topology from $\partial\free\times\partial\free$. A \define{lamination} on $\free$ is a nonempty closed subset of $\db\free$ that is invariant under both the flip map $(\eta,\xi)\mapsto(\xi,\eta)$ and the (diagonal) action of $\free$. A lamination is \define{minimal} if it does not contain a proper sublamination. If $L$ is a lamination on $\free$, we write $L'$ to denote the set of accumulation points of $L$ in $\db\free$. Note that $L'\subset L$, since $L$ is closed, and that $L'$ is itself a lamination.

Following \cite{KapCurrents}, a \define{geodesic current} on $\free$ is a positive Radon measure on $\db\free$ that is both flip and $\free$--invariant. Notice that the support $\supp(\mu)$ of every nonzero current $\mu$ is necessarily a lamination. We write $\curr$ for the space of all geodesic currents on $\free$ equipped with the weak topology.  Quotienting by the action of $\R_+$ by scalar multiplication yields the compact space $\pcurr$ of \define{projective geodesic currents}. 

Let us discuss some basic examples of these concepts. Every nontrivial $a\in \free$ acts on $\partial \free$ with a unique attracting fixed point $a^+$ and repelling fixed point $a^-$. The \define{lamination of a (nontrivial) conjugacy class} $\alpha$ of $\free$ is then defined to be
\[L(\alpha)\colonequals \bigcup_{a\in \alpha}\{(a^+,a^-),(a^-,a^+)\};\]
notice that $L(\alpha)$ is indeed a closed and $\free$--invariant subset of $\db\free$. Correspondingly, the \define{counting current} of a (nontrivial) conjugacy class $\alpha = \beta^m$, where $\beta$ is not a proper power, is defined as
\[\eta_\alpha \colonequals m\eta_\beta \colonequals m\sum_{b\in\beta} \delta_{(b^-,b^+)}+\delta_{(b^+,b^-)}.\]
One may check that $\eta_\alpha$ is indeed a current and moreover that $\supp(\eta_\alpha) = L(\alpha)$.

\subsection{Trees}
\label{sec:trees}

An \define{$\R$--tree} is a $0$--hyperbolic geodesic metric space. Alternately, an $\R$--tree is a metric space in which there is a unique embedded path between any two points and this path is a geodesic. 
Throughout this paper, we will use the term \define{tree} to mean an $\R$--tree equipped with an isometric and minimal action of $\free$; a tree is \define{minimal} if it does not contain a proper $\free$--invariant subtree.

We write $\len_T(a)$ for the \define{translation length} of an element $a\in \free$ acting on a tree $T$, that is, $\len_T(a) = \inf_{t\in T}d(t,a\cdot t)$. Notice that $\len_T(a)$ depends only on the conjugacy class $\alpha$ of $a$. The element $a$ acts \define{hyperbolically} (with an invariant axis) on $T$ if $\len_T(\alpha) >0$ and \define{elliptically} (with a nonempty fixed subtree) if $\len_T(\alpha) = 0$. The tree $T$ is said to be \define{free} if $\len_T(\alpha) > 0$ for all nontrivial conjugacy classes $\alpha$.

Coulbois, Hilion and Lustig \cite{CHL2} have associated to every tree $T$ a \define{dual lamination}
\[L(T)\colonequals \bigcap_{\epsilon > 0}\left(\overline{\bigcup_{\alpha\in \Omega_\epsilon(T)} L(\alpha)}\right)\subset \db\free,\]
where $\Omega_\epsilon(T) = \{\alpha : \len_T(\alpha)<\epsilon\}$ is the set of conjugacy classes with short translation length in $T$ and the closure is taken in $\db\free$. Observe that $L(\alpha)\subset L(T)$ if and only if $\len_T(\alpha) = 0$; thus $T$ is free if and only if $L(T)$ does not contain the lamination $L(\alpha)$ of any nontrivial conjugacy class $\alpha$. The set $L(T)$ is nonempty, and thus a bona fide lamination, unless $T$ is free and simplicial.  We refer the reader to \cite{CHL2} for a more detailed discussion of $L(T)$.

We say that a tree is \define{very small} if the stabilizer of any segment of $T$ is maximal cyclic and the stabilizer of every tripod is trivial \cite{CohenLustigVS}.
A tree $T$ is said to have \define{dense orbits} if every $\free$ orbit is dense in $T$. At the other extreme, if every orbit is discrete then the tree is said to be \define{simplicial}. 
A tree is \define{arational} if there does not exist a proper free factor $A$ of $\free$ and an $A$--invariant subtree on which $A$ acts with dense orbits. Following Guirardel \cite{Guiradel-trees}, we say that a tree $T$  is \define{indecomposable} if for every pair of nondegenerate arcs $\tau,\tau'\subset T$ there exist $a_1,\dotsc,a_n\in \free$ so that $\tau'\subset a_1\tau\cup\dotsb\cup a_n\tau$ with $a_i\tau\cap a_{i+1}\tau$ nondegenerate for each $1 \le i < n$; indecomposablity is thus a strong mixing property for the action of $\free$ on $T$. 
The following theorem of Reynolds clarifies the relationship between these notions: 

\begin{theorem}[Reynolds \cite{Rey12}]\label{th:char_arational}
A minimal, very small tree is arational if and only if it is indecomposable and either $(1)$ free or $(2)$ dual to a filling measured lamination on a once-punctured surface.
\end{theorem}

\subsection{The free factor complex}

A nontrivial subgroup $A\le \free$ of $\free$ is a \define{free factor} of $\free$ if there exists a complementary nontrivial subgroup $B\le \free$ such that $\free = A\ast B$. As is common we often blur the distinction between free factors and their conjugacy classes. The \define{(free) factor complex} of $\free$ is the simplicial complex $\fc$
whose $k$--simplices consist of chains $A_0 < \dotsb < A_k$ of properly nested (conjugacy classes of) free factors of $\free$
and whose face inclusions correspond to subchains. 
Note that $\fc$ is not locally compact, and that the group $\Out(\free)$ acts on $\fc$ by simplicial automorphisms.
 We equip $\fc$ with the path metric in which simplices are all isometric to standard Euclidean simplices with side lengths equal to $1$; the induced path metric on the $1$--skeleton makes $\fc^1$ into a simplicial metric graph with all edges having length $1$. 
With this setup we have the following foundational result of Bestvina and Feighn:

\begin{theorem}[Bestvina--Feighn \cite{BFhyp}]
\label{thm:fc_hyperbolic}
The factor complex $\fc$ is Gromov hyperbolic.
\end{theorem}

As the full complex $\fc$ and simplicial graph $\fc^1$ are quasi-isometric, we henceforth work exclusively with the  $1$--skeleton. 
In particular, for $A,B\in \fc^0$ we write $d_\fc(A,B)$ to mean the distance from $A$ to $B$ in the path space $\fc^1$.

\subsection{Outer space}
\label{sec:outer_space}

Let $\rose$ be the $\rank(\free)$--petal rose with base vertex $v$, and fix an isomorphism $\pi_1(\rose,v)\cong\free$. A \define{core graph} is a finite $1$--dimensional CW-complex $G$ with no valence $1$ vertices;  
and by a \define{metric} on $G$ we mean a path metric for which the interior of each $1$--cell (with the induced path metric) is isometric to a 
positive-length 
open subinterval of $\R$. 
The \define{volume} of $G$ is the sum of its edge lengths, and a \define{marking} of $G$ is a homotopy equivalence $g\colon\rose\to G$. Culler and Vogtmann's \cite{CVouter}  \define{unprojectivized Outer space} of $\free$ is the space $\cv = \{(G,g)\}$ of marked metric core graphs, modulo the equivalence relation $(G,g)\sim(H,h)$ if there exists an isometric \define{change of marking map} $G\to H$ in the homotopy class $h\circ g\inv$. By (projectivized) \define{Outer space} $\os$, we simply mean the subset of $\cv$ consisting of volume $1$ marked metric graphs. 
We equip $\os$ with the \define{asymmetric metric $d_\os$} defined as follows:
\begin{equation*}\label{eqn:lip_metric}
d_\os((G,g),(H,h)) \colonequals \inf \left\{\log(\Lip(\phi)) \mid \phi\simeq h\circ g\inv\right\}
\end{equation*}
where $\Lip(\phi)$ denotes the optimal Lipschitz constant for the change of marking map $\phi\colon G\to H$.
The \define{symmetrization} $\dsym(G,H) = d_\os(G,H)+d_\os(H,G)$ is an honest metric and defines the topology on $\os$.
We will suppress the marking and metric and denote points in $\cv$ and $\os$ simply by the underlying graph.

Given any subgroup $A\le\free$ (or conjugacy class thereof) and point $G\in \cv$, we write $A\vert G$ for the maximal core subgraph of the cover of $G$ corresponding to $A$ and we equip this with the pull-back metric from the immersion $A\vert G\to G$. When convenient we will blur the distinction between the metric core graph $A\vert G$ and the immersion $A\vert G\to G$ itself.
For a conjugacy class $\alpha$ and write $\len(\alpha\vert G)$ for the volume of the graph $\alpha\vert G$; this is the \define{length of $\alpha$ at $G$}. With this notation we have the following useful formula for the metric \cite{FMout}:
\[d_\os(G,H) = \log \left(\sup_{\alpha\in \free} \frac{\len(\alpha\vert H)}{\len(\alpha\vert G)}\right).\]

Observe that the universal cover $\tilde{G}$ of a graph $G\in \cv$ is naturally a simplicial $\R$--tree equipped with an action of $\free\cong\pi_1(G)$ by deck transformations (where the isomorphism $\free\cong \pi_1(G)$ is provided by the marking $\rose\to G$). With this perspective $\len(\alpha\vert G)$ is simply the translation length $\len_{\tilde G}(\alpha)$ of the conjugacy class $\alpha$ on the tree $\tilde{G}$. In fact, this correspondence gives a bijection between $\cv$ and the set of free simplicial $\R$--trees up to $\free$--equivariant isometry.

\paragraph{Asymmetry in Outer space.}
Care must be taken to cope with the asymmetry inherent in Outer space. For us a \define{geodesic in $\os$} always means a \emph{directed} geodesic, that is, a map $\gamma\colon \I\to \os$ so that $d_\os(\gamma(s),\gamma(t)) = t-s$ for all $s < t$. Similarly a \define{$K$--quasigeodesic} is a map $\gamma\colon \I\to \os$ so that
\[\tfrac{1}{K}(t-s)-K \le d_\os(\gamma(s),\gamma(t)) \le K(t-s)+K\]
for all $s < t$.
On the other hand, for $r > 0$ the \define{$r$--neighborhood} $\nbhd{r}{U}$ of a subset $U\subset \os$ must be defined using the symmetrized metric:
\[\nbhd{r}{U}\colonequals \{G\in \os\mid \dsym(G,u)\le r\text{ for some }u\in U\}.\]
The \define{Hausdorff distance} $\dhaus(U,W)$ between two subsets $U,W\subset\os$ is then defined, as usual, to be the infimal $r$ so that $U\subset \nbhd{r}{W}$ and $W\subset \nbhd{r}{U}$. 
For $\epsilon> 0$, we write
\[\os_\epsilon \colonequals \{G\in \os \mid \len(\alpha\vert G) \ge \epsilon\text{ for all }\alpha\in \free\setminus 1\}\]
for the \define{$\epsilon$--thick part} of Outer space. The following important result bounds the asymmetry in $\os_\epsilon$.

\begin{lemma}[Handel--Mosher \cite{handel2007expansion}, Algom-Kfir--Bestvina \cite{AlBest}]
\label{lem:asymmetry}
For every $\epsilon > 0$ there exists $\sym_\epsilon \ge 1$ so that for all $G,H\in \os_\epsilon$ one has
\[d_\os(H,G) \le \dsym(H,G) = \dsym(G,H) \le \sym_\epsilon \cdot d_\os(G,H).\]
\end{lemma}

\paragraph{Projecting to the factor complex.}
There is a coarse projection $\fproj\colon \os\to \fc$ defined by sending $G\in \os$ to the set of free factors $\pi_1(G')$ corresponding to proper, connected, noncontractible subgraphs $G'$ of $G$ (here $\pi_1(G')$ is identified with a free factor of $\free$ by the marking $\rose\to G$). One may easily check that $\diam_\fc(\fproj(G))\le 4$ \cite[Lemma 3.1]{BFhyp}, so we are justified in viewing $\fproj$ as a coarse projection.  For $G,H\in \os$ we define
\[d_\fc(G,H) = \diam_\fc(\fproj(G)\cup\fproj(H)).\]
The following appears as Lemma 2.9 in \cite{DT1} and follows from \cite[Corollary 3.5]{BFhyp}.
\begin{lemma}
\label{lem:fproj_lipschitz}
For all $G,H\in \os$ we have $d_\fc(G,H)\le 80d_\os(G,H)+80$.
\end{lemma}

The projection $\fproj$ provides an important connection between the geometries of $\os$ and $\fc$. For example, the following stability result uses the geometry of $\fc$ to establish aspects of hyperbolicity in $\os$ and served as a main tool in our proof of \Cref{th:DT1}.

\begin{theorem}[Dowdall--Taylor \cite{DT1}]\label{th:DT_2}
Let $\gamma\colon \I\to \os$ be a $K$--quasigeodesic whose projection $\pi \circ \gamma\colon \I\to \fc$ is also a $K$--quasigeodesic. Then there exist constants $A,\epsilon > 0$ and $K'\ge 1$ depending only on $K$ (and the injectivity radius of the terminal endpoint $\gamma(\Ipl)$ when $\Ipl<\infty$) with the following property: If $\rho\colon \J\to \os$ is any geodesic with the same endpoints as $\gamma$, then
\begin{itemize}
\item[(i)] $\gamma(\I), \rho(\J)\subset \os_{\epsilon}$,
\item[(i)] $d_{\mathrm{Haus}}(\gamma(\I),\rho(\J)) < A$, and
\item[(ii)] $\pi\circ \rho\colon \J\to \fc$ is a (parameterized) $K'$--quasigeodesic.
\end{itemize} 
\end{theorem}

\paragraph{Folding.} 
\label{sec:folding}
We will need a particular class of directed geodesics in $\os$ called folding paths, which we now briefly describe.  A \define{segment} in a metric core graph $G$ is a locally isometric immersion of an interval $[0,L]$ into $G$, and a \define{direction} at $p\in G$ is a germ of nondegenerate segments with $0\mapsto p$. A \define{turn} is an unordered pair $\{d,d'\}$ of distinct directions at a vertex of $G$. 

A map $\phi\colon G\to H$ of metric core graphs that is a local $\Lip(\phi)$--homothety induces a derivative map $D_\phi$ which sends a direction at $p$ to a direction at $\phi(p)$. Two directions at $p$ are said to be in the same \define{gate} if they are identified by $D_\phi$. The map then gives rise to an \define{illegal turn structure} on $G$, whereby a turn $\{d,d'\}$ is \define{illegal} if $d$ and $d'$ are in the same gate and is \define{legal} otherwise. 

We say that a map $\phi\colon G\to H$ between points $G,H\in \os$ is a \define{folding map} if it is homotopic to the change of markings $h\circ g\inv$, is a local $\Lip(\phi)$--homothety, satisfies $d_\os(G,H) = \log \Lip(\phi)$, and it induces at least $2$ gates at each point $p\in G$.
As described in \cite[\S2]{BFhyp}, each folding map $\phi\colon G\to H$ gives rise to a unique \define{folding path} $\{G_t\}_{t\in [0,L]}$ in $\os$ via first \emph{folding all illegal turns at speed one} and second \emph{rescaling} to obtain graphs in $\os$. The folding path comes equipped with a family of folding maps $\{\phi_{st}\colon G_s\to G_t\}_{s\le t}$ satisfying
\[ \text{$\phi_{0L}=\phi$,\, $\phi_{ss} = \mathrm{Id}_{G_s}$,\, and $\phi_{rt} = \phi_{st}\circ\phi_{rs}$ with $\phi_{rs}$ and $\phi_{rt}$ inducing the same illegal turn structure on $G_r$}\]
for all $0\le r \le s\le t\le L$. 
Our convention is to parameterize the folding path $\{G_t\}_{t\in [0,L]}$ with respect to arc length in Outer space. With this convention, the assignment $t\mapsto \gamma(t) = G_t$ defines a directed geodesic $\gamma\colon [0,L]\to \os$ from $G= G_0$ to $H = G_L$.
See Proposition 2.2 and Notation 2.4 of \cite{BFhyp} or \cite[\S2]{DT1} for the details of this construction; further properties will be recalled in \S\ref{sec:more_folding}.

\subsection{The boundaries of Outer space and the factor complex}
\label{sec:os_boundary}

The length functions give an embedding of unprojectivized Outer space $\cv$ into $\R^\free$ via $G\mapsto (\len(a \vert G))_{a \in \free}$, and thus $\cv$ inherits the subspace topology from $\R^\free$. The resulting topology on $\os\subset \cv$ agrees with the one induced by the symmetrized metric $\dsym$. The work of Cohen--Lustig \cite{CohenLustigVS} and Bestvina--Feighn \cite{BF-outer_limits} shows that the closure $\cvbar$ of $\cv$ in $\R^\free$ may be identified with the space of minimal very small trees. Projectivizing, one similarly identifies the closure $\osbar$ of $\os\subset\mathbb{P}\R^\free$ with the space of projective classes of minimal very small trees. The \define{boundary of $\os$} is consequently defined to be the set $\partial \os \colonequals \osbar \setminus \os$ of projective classes very small trees that are not both free and simplicial.

In \cite{kapovich2007geometric}, Kapovich and Lustig introduced an $\Out(\free)$--invariant intersection pairing $\langle\cdot,\cdot\rangle$ between very small minimal trees and currents; we record here a few of its properties:

\begin{theorem}[Kapovich--Lustig \cite{kapovich2007geometric,KapLust-LamCurrs}]
\label{th:currents}\label{th:supp_of_current}
There is a unique $\Out(\free)$--invariant continuous pairing
\[
\langle\cdot, \cdot\rangle\colon \cvbar \times \curr \to \R_+
\]
which is homogeneous in the first coordinate and linear in the second. Moreover, for every tree $T$, current $\mu$, and conjugacy class $\alpha$, we have that $\langle T,\eta_\alpha \rangle = \len_T(\alpha)$ and that $\langle T,\mu\rangle = 0$ if and only if $\supp(\mu)\subset L(T)$. 
\end{theorem}

\noindent We will be particularly interested in the case where the tree $T$ is free and indecomposable.
\begin{theorem}[Coulbois--Hilion--Reynolds {\cite[Corollary 1.4]{CHR11}}]
\label{th:CHR}
If $T\in \cvbar$ is free and indecomposable, then $\langle T,\mu\rangle = 0$ if and only if $\supp(\mu) = L'(T)$. 
\end{theorem}

Finally, let $\arational$ be the subspace of $\partial cv$ consisting of arational trees. For $T,'T \in \arational$, say $T\sim T'$ if $L(T)=L(T')$. The following result computes the boundary of $\F$.

\begin{theorem}[Bestvina--Reynolds \cite{bestvina2012boundary}, Hamenst\"adt \cite{hamenstadt2013boundary}]\label{th:factor_complex_boundary}
The map $\fproj \colon \X \to \F$ has a continuous extension to a map $\partial\fproj \colon \arational \to \partial \F$, in the sense that if $G_i \to T$ in $\osbar$ and $T\in \arational$, then $\fproj(G_i) \to \partial\fproj(T)$ in $\fc\cup\partial\fc$. Moreover, if $T \sim T'$ then $\partial\fproj(T) = \partial\fproj(T')$, and the induced map $\arational / \sim \to \partial \fc$ is a homeomorphism.
\end{theorem}

\section{Alignment preserving maps and boundaries}
\label{sec:alignmentpres}

In \S\ref{sec:co-surface graph} we will introduce and analyze a new $\Out(\free)$--graph termed the co-surface graph. For this analysis, we develop a general framework for computing the boundary of any space whose hyperbolicity may be obtained by the Kapovich--Rafi method \cite{KapRaf}. As this result is applicable in other contexts, we state it in general terms.

Let $X$ and $Y$ be geodesic metric spaces. We say that three (ordered) points $a,b,c\in X$ are \define{$K$--aligned} if $d_X(a,b) + d_X(b,c) \le d_X(a,c) +K$; the points are simply said to be \define{aligned} if they are $0$--aligned.  
We say that a Lipschitz map $p\colon X\to Y$ is \define{alignment preserving} if there exists $K\ge 0$ such that $p(a),p(b),p(c)$ are $K$--aligned whenever $a,b,c$ are aligned. 

\begin{lemma}
\label{lem:generalized_alignment}
Suppose that $q\colon W\to X$ and $p\colon X\to Y$ are alignment preserving maps between geodesic metric spaces and that $X$ is $\delta$--hyperbolic. Then for all $L\ge 0$ there is an $L'\ge 0$ such that $p(a),p(b),p(c)$ are $L'$--aligned whenever $a,b,c$ are $L$--aligned. Moreover, the composition $p\circ q\colon W\to Y$ is alignment preserving.
\end{lemma}

Let us formalize some observations that will aid in the proof. Firstly, every $\delta$--thin triangle admits a \define{$2\delta$--barycenter}, meaning a point $\omega$ that lies within $2\delta$ of each side of the triangle (and is consequently $4\delta$--aligned between any two vertices). Secondly, whenever the three triples $(a,\omega,b)$, $(b,\omega,c)$, and $(c,\omega, a)$ are each $K$--aligned, the triangle inequality immediately gives
\begin{equation}\label{eqn:Gproducts_and_centers}
\abs{d(a,\omega) - (b\vert c)_a} \le K.
\end{equation}
Combining these, we see that in a $\delta$--hyperbolic space $X$, the Gromov product $(b\vert c)_a$ lies within $4\delta$ of $d_X(a,\omega)$ for any $2\delta$--barycenter $\omega$ of the geodesic triangle $\triangle(a,b,c)$. 

\begin{proof}
Let $K\ge 0$ be the alignment constant of the alignment preserving map $p$. Take $L$--aligned points $a,b,c$ in $X$ and let $\omega$ be a $2\delta$--barycenter for the geodesic triangle $\triangle(a,b,c)$. Then  $d_X(b,\omega)$ is within $4\delta$ of $(a\vert c)_b$, which is in turn bounded by $L/2$ since $a,b,c$ are $L$--aligned. Thus there is a point $x\in [a,c]$ with $d_X(b,x)\le 6\delta + L$. Then $p(a),p(x),p(c)$ are $K$--aligned, and so the points $p(a),p(b),p(c)$ are $L'$--aligned, where $L'$ is $K$ plus $2(6\delta +L)$ times the Lipschitz constant for $p$. This proves the first claim. The second claim now follows immediately from the first.
\end{proof}

If $X$ and $Y$ are hyperbolic and $p\colon X\to Y$ is alignment preserving, the \define{$Y$--subboundary of $X$} (relative to $p$) is defined to be
\[
\partial_Y X = \{ \gamma(\infty)\in \partial X \mid \gamma\colon\R_+\to X\text{ is a quasigeodesic ray with }\diam_Y p(\gamma(\R_+)) = \infty \}.
\]
Informally, $\partial_Y X$ consists of those points in $\partial X$ that ``project to infinity'' in $Y$. This is made precise by the following theorem.

\begin{theorem}[Boundaries]
\label{th:boundaries}
Suppose that $p\colon X \to Y$ is a coarsely surjective, alignment preserving map between hyperbolic spaces. Then $p$ admits an extension to a homeomorphism $\partial p \colon \partial_Y X \to \partial Y$. Moreover, the extension $p\cup\partial p\colon X \cup \partial_Y X \to Y \cup \partial Y$ is continuous in the sense that if $x_n \to \lambda \in \partial_YX$ as $n \to \infty$, then $p(x_n) \to \partial p(\lambda) \in \partial Y$.
\end{theorem}
\begin{proof}
Let $\delta$ be the hyperbolicity constant of $X$, let $L$ be the Lipschitz constant for $p$, and let $K'\ge 0$ be the constant, provided by \Cref{lem:generalized_alignment}, such that $p(a),p(b),p(c)$ are $K'$--aligned whenever $a,b,c\in X$ are $4\delta$--aligned. We then have the following useful observation: If $\omega$ is any $2\delta$--barycenter for an arbitrary triangle $\triangle(a_1,a_2,a_3)$ in $X$, then the triples $(p(a_i),p(\omega),p(a_j))$ for distinct $i,j\in \{1,2,3\}$ are each $K'$--aligned and so we may apply \eqref{eqn:Gproducts_and_centers} in both $X$ and $Y$ to conclude that
\begin{equation}
\label{eqn:lip_gromov_prod}
(p(a_2)\vert p(a_3))_{p(a_1)} \le K'+d_Y(p(a_1),p(\omega)) \le K' + Ld_X(a_1,\omega) \le L(a_2\vert a_3)_{a_1} + 4\delta+K'.
\end{equation}

To define the map $\partial p$, choose a quasigeodesic ray $\gamma\colon \R_+\to X$ with $\diam_Y p(\gamma(\R_+)) = \infty$ and consider the admissible sequence $\{a_n\}_{n=0}^\infty$, where $a_n=\gamma(n)$. Set $b_n = p(a_n)$. Since ordered triples of points along $\gamma(\R_+)$ are uniformly aligned by \Cref{prop:stability_of_quasis}, the assumption $\diam_Y p(\gamma(\R+)) = \infty$ in fact implies that $\lim_t d_Y(b_0,p(\gamma(t))) = \infty$. For each pair $n,m\ge 0$, choose a $2\delta$--barycenter $c_{n,m}$ for the triangle $\triangle(a_0,a_n,a_m)$. Then $(a_n\vert a_m)_{a_0}$ is within $4\delta$ of $d_X(a_0,c_{n,m})$ by \eqref{eqn:Gproducts_and_centers}. By \Cref{prop:stability_of_quasis}, $c_{n,m}$ also lies within uniformly bounded distance of $\gamma(t_{n,m})$ for some $t_{n,m}\in \R_+$. By admissibility and the fact that $\gamma$ is a quasigeodesic, the quantities $d_X(a_0,c_{n,m})$ and $t_{n,m}$ both tend to infinity as $n,m\to\infty$. Therefore $\lim_{n,m} d_Y(b_0,p(c_{n,m}))=\infty$ since $p$ is Lipschitz. However, $(b_n\vert b_m)_{b_0}$ is within $K'$ of $d_Y(b_0,p(c_{n,m}))$ because $p(c_{n,m})$ is $K'$--aligned between the three points $b_0$, $b_n$ and $b_m$. Consequently $\{b_n\}$ is admissible, and we may define $\partial p(\gamma(\infty))$ to be $\{b_n\}\in \partial Y$.

We now prove that $\partial p$ is well-defined and that $p\cup \partial p$ is continuous. Let $\lambda = \gamma(\infty)\in \partial_Y X$ with $\gamma$, $\{a_n\}$, and $\{b_n\}$ as above. Suppose that $\{x_n\}$ is a sequence in $X$ converging to $\lambda$. This simply means that $\{x_n\}$ is admissible and equivalent to $\{a_n\}$. Letting $e_{n,m}$ denote a $2\delta$--barycenter for $\triangle(a_0,a_n,x_m)$, we have that $d_X(a_0,e_{n,m})$ is within $4\delta$ of $(a_n\vert x_m)_{a_0}$ and thus tends to infinity. This barycenter $e_{n,m}$ is also uniformly close, again by \Cref{prop:stability_of_quasis}, to some point $\gamma(s_{n,m})$ with $s_{n,m}$ necessarily tending to infinity since $\gamma$ is a quasigeodesic. As before, it follows that $d_Y(b_0,p(e_{n,m}))\to \infty$ and, since $p$ is alignment preserving, that this quantity coarsely agrees with $(p(a_n)\vert p(x_m))_{b_0}$. Thus $\{p(x_n)\}$ is equivalent to $\{p(a_n)\}$, proving that $\{p(x_n)\}$ is admissible and converges to $\partial p(\lambda) = \{b_n\}$. In particular, for any quasigeodesic $\gamma'\colon \R_+\to X$ with $\gamma'(\infty)=\lambda$, it follows that $\{p(\gamma'(n))\}$ is equivalent to $\{b_n\} = \{p(\gamma(n))\}$. Thus $\partial p$ is well-defined and the extension $p\cup \partial p$ is continuous in the manner claimed.

We next show $\partial p$ is injective. Suppose $\lambda,\mu\in \partial_Y X$ satisfy $\partial p(\lambda) = \partial p(\mu)$. 
If $x_n,z_n\in X$ are any sequences with $x_n\to \lambda$ and $z_n\to \mu$, then $\{p(x_n)\}$ and $\{p(z_n)\}$ are equivalent by the continuity of $p\cup \partial p$. Therefore $(p(x_n)\vert p(z_m))_{p(x_0)}\to \infty$ which, by (\ref{eqn:lip_gromov_prod}), forces $(x_n\vert z_m)_{x_0}\to \infty$ as well. Thus $\lambda=\mu$ and $\partial p$ is injective.

\begin{figure}[htbp]
\labellist
\small\hair 3pt
\pinlabel $\lambda=\gamma(\infty)$ [l] at 445 105
\pinlabel $\gamma$ [b] at 118 100
\pinlabel $\gamma(0)=x_0$ [r] <1pt,0pt> at 6 75
\pinlabel $\gamma(m)$ [b] at 389 106
\pinlabel $\gamma(t)$ [b] at 184 91
\pinlabel $x_m$ [l] at 385 30
\pinlabel $x_k$ [l] at 196 4
\pinlabel $e_{k,m}$ [t] <1.5pt,1pt> at 176 50
\pinlabel $c_m$ [r] <1pt,0pt> at 342 77
\pinlabel $a$ [br] <1pt,-1pt> at 342 90
\pinlabel $e'$ [tl] <0pt,1pt> at 178 67
\pinlabel $c'$ [bl] <1pt,-1.5pt> at 336 59
\endlabellist
\begin{center}
\includegraphics[width =.7 \textwidth]{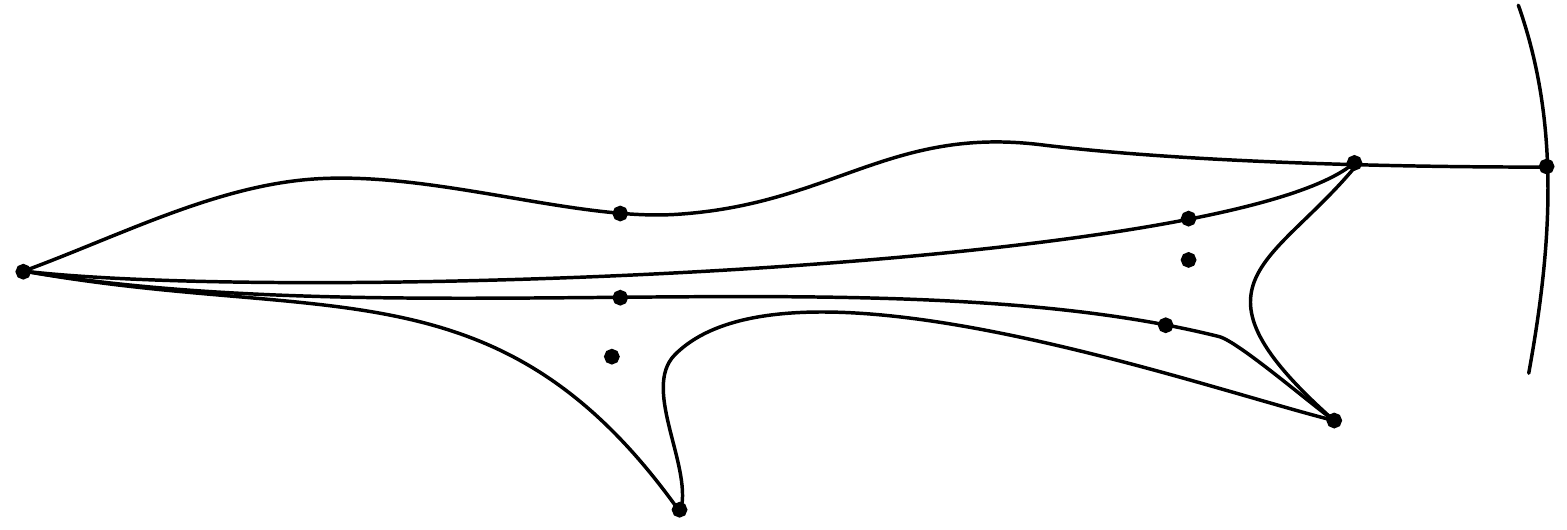}
\caption{Proving surjectivity of $\partial p$.}
\label{Fig:surjective_p}
\end{center}
\end{figure}
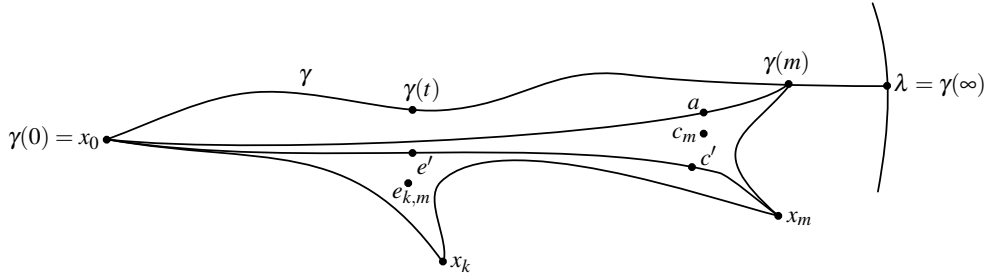

To see surjectivity, let $\eta\in \partial Y$. For a sequence $y_n\in Y$ with $y_n\to \eta$, choose points $x_n\in X$ so that $d_Y(p(x_n),y_n)$ is uniformly bounded and thus $p(x_n)\to \eta$ as well. 
Since $\{p(x_n)\}$ is admissible \eqref{eqn:lip_gromov_prod} implies that $\{x_n\}$ is admissible and so converges to some point $\lambda\in \partial X$. Now build $10\delta$--quasigeodesic $\gamma\colon \R_+\to X$ with $\gamma(0)=x_0$ and $\gamma(\infty) = \lambda$. Refer to \Cref{Fig:surjective_p}.
Also let $e_{n,m}$ be a $2\delta$--barycenter of $\triangle(x_0,x_n,x_m)$.
We claim that the projection of $\gamma$ to $Y$ has infinite diameter. 
To see this, fix any $D > 0$ and use admissibility of $\{p(x_n)\}$ to find $N$ so that $d_Y(p(x_0),p(e_{n,m})) > D$ for all $n,m\ge N$. Fix some $k\ge N$. Using the equivalence of $\{x_n\}$ and $\{\gamma(n)\}$, we may then choose $m\ge N$ so that $d_X(x_0,c_m) \ge d_X(x_0,x_k)+8\delta$, where $c_m$ denotes a $2\delta$--barycenter for $\triangle(x_0,x_m,\gamma(m))$. Since the triple $(x_0,e_{k,m},x_k)$ is $4\delta$--aligned, this implies $d_X(x_0,c_m)\ge d_X(x_0,e_{k,m})+4\delta$. Choosing points $e'$ and $c'$ in $[x_0,x_m]$ within $2\delta$ of $e_{k,m}$ and $c_m$, respectively, it follows that $d_X(x_0,c')\ge d_X(x_0,e')$ so that $e'\in [x_0,c']$. Now, since $c_m$ is a barycenter for $\triangle(x_0,x_m,\gamma(m))$, there is a point $a\in [x_0,\gamma(m)]$ with $d_X(c',a)\le 4\delta$. By thinness of the triangle $\triangle(x_0,c',a)$, we see that $e'\in [x_0,c']$ lies within $5\delta$ of $[x_0,a]\subset [x_0,\gamma(m)]$. Therefore, since $d_X(e_{k,m},e')\le 2\delta$ and $\gamma$ is a quasigeodesic, \Cref{prop:stability_of_quasis} and the triangle inequality imply that $e_{k,m}$ is within $7\delta + R_0(10\delta,\delta)$ of some point $\gamma(t)$. Consequently, $d_Y(p(x_0),p(\gamma(t)))$ coarsely agrees with $d_Y(p(x_0),p(e_{k,m}))\ge D$ up to uniformly bounded additive error. Since $D$ here is arbitrary, it follows that $\diam_Y(p(\gamma)) = \infty$. Thus $\lambda\in \partial_YX$, which proves that $\partial p$ is surjective.

Finally we prove $\partial p$ is a homeomorphism. Firstly, for any $\lambda,\mu\in \partial_YX$ and $x\in X$, \Cref{eqn:lip_gromov_prod} and our proof of surjectivity show that $(\partial p(\lambda)\vert \partial p(\mu))_{p(x)} \le L (\lambda\vert \mu)_x + 4\delta + K'$. By the definition of the topology on the Gromov boundary, it immediately follows that $\partial p \colon \partial_Y X\to \partial Y$ is open. Conversely, for any $\eta\in \partial Y$ and $D > 0$, we may, as above, build a $10\delta$--quasigeodesic $\gamma\colon \R_+\to X$ converging to $\lambda = (\partial p)^{-1}(\eta)$ and whose projection to $Y$ has infinite diameter. Set $x=\gamma(0)$. Thus there is some $R$ so that $d_Y(p(x),p(z))\ge D+K'$ for all $z\in X$ that lie within $2\delta$ of $\gamma$ and satisfy $d_X(x,z)\ge R$. Now if $\mu\in \partial_Y X$ is such that $(\mu\vert \lambda)_{x}\ge R+4\delta$, then we may choose a sequence $\{z_n\}$ converging to $\mu$ so that
\[\liminf_{n,m\to\infty} (z_n\vert \gamma(m))_{x} \ge R+4\delta.\]
Thus if $e_{n,m}$ is a $2\delta$--barycenter for $\triangle(x,z_n,\gamma(m))$, then $\liminf_{n,m} d_X(x,e_{n,m})\ge R$. But since $e_{n,m}$ lies within $2\delta$ of $\gamma$, we have that $d_Y(p(x),p(e_{n,m}))\ge D+K'$ and thus also $(p(z_n)\vert p(\gamma(m))_{p(x)}\ge D$ for all large $n,m$. This proves that $(\partial p(\mu)\vert \eta)_{p(x)} \ge D$ for all $\mu$ satisfying $(\mu\vert \lambda)_{x}\ge R+4\delta$. Therefore $\partial p$ is continuous. 
\end{proof}

\begin{remark}
Note that if the hypothesis of coarse surjectivity in \Cref{th:boundaries} is dropped, the proof shows that the map $\partial p \colon \partial_Y X \to \partial Y$ is a topological embedding. 
\end{remark}

We also record the following useful lemma, the idea of which is well-known to experts (see for example \cite[Lemma 2.6]{HamStability}). First say that $p \colon X \to Y$ is \define{metrically proper} if for any $D\ge0$ there is a $C\ge0$ so that $d_X(a,b) \ge C$ implies $d_Y(p(a),p(b)) \ge D$.

\begin{lemma}\label{lem:CAP} 
Suppose that $X$ and $Y$ are geodesic metric spaces.  
If $p \colon X \to Y$ is alignment preserving and metrically proper, then $p$ is a quasi-isometric embedding.
\end{lemma}

\begin{proof}
Since $p \colon X \to Y$ is alignment preserving, there is a constant $K$ such that whenever $a,b,c\in X$ are $0$--aligned we have 
\[d_Y(p(a),p(c)) \ge d_Y(p(a),p(b)) + d_Y(p(b),p(c)) - K.\]
Further, since $p \colon X \to Y$ is metrically proper, there is a $C>0$ such that if $a,b\in X$ satisfy $d_X(a,b)\ge C$, then $d_Y(p(a),p(b)) \ge 2K$. Now let $x$ and $x'$ be points of $X$ with $d_X(x,x') = d$ and let $\gamma\colon [0,d] \to X$ be a geodesic with $\gamma(0) = x$ and $\gamma(d) = x'$. Let $N$ be the largest integer less than $\frac{d}{C}$, and set $a_i = \gamma(iC)$ for $0\le i \le N$. Then,
\begin{align*}
d_Y(p(x),p(x')) &\ge d_Y(p(a_0), p(a_N)) + d_Y(p(a_N),p(x')) - K\\
&\ge  d_Y(p(a_N),x')-K +\sum_{i=0}^{N-1} (d_Y(p(a_i),p(a_{i+1})) - K)\\
&\ge -K + K \cdot N \\
&\ge \frac{K}{C} \cdot d_X(x,x')-2K.
\end{align*}
Since $p \colon X \to Y$ is Lipschitz by assumption, this completes the proof.
\end{proof}

\section{The co-surface graph $\CS$}
\label{sec:co-surface graph}
This section introduces the co-surface graph $\CS$ of the free group $\free$ and develops its basic properties. First we define $\cs$ and discuss its relationship to other $\Out(\free)$--graphs appearing in the literature. Then in \S\ref{sec:cosurfaceboundary} we use the theory of alignment preserving maps to calculate the boundary of $\cs$. Finally, in \S\ref{sec:qi_into_cosurf} we show that a subgroup $\Gamma\le\Out(\free)$ qi-embeds into $\cs$ if and only if it is purely atoroidal and qi-embeds into $\fc$.

Recall that an element $a \in \free$ is \define{primitive} is $a$ belongs to a free basis for $\free$; this is an invariant of the conjugacy class $\alpha$ of $a$ and so we also call $\alpha$ primitive. The \define{primitive loop graph} of $\free$ is the simplicial graph $\pl$ whose vertices are the primitive conjugacy classes of $\free$, and where two conjugacy classes $\alpha,\beta$ are joined by an edge if and only if they have representatives that are jointly part of a free basis of $\free$. We equip $\pl$ with the path metric $d_\pl$ in which each edge has length $1$. As each primitive element generates a cyclic free factor of $\free$, there is natural inclusion map $\pl^0\to\fc$. It is straightforward to check that this inclusion is $2$--bilipschitz and $1$--dense and therefore admits a $4$--quasi-isometry coarse inverse which we denote $\mathcal{D}\colon \fc\to \pl$. We also have the coarse projection $\plproj\colon \os\to \pl$ defined by sending $G\in \os$ to the set of embedded closed loops on $G$; this projection coarsely agrees with the composition $\mathcal{D}\circ\fproj$.

The primitive loop graph measures, in a sense, how algebraically complicated primitive conjugacy classes are with respect to each other. The co-surface graph, on the other hand, is designed to measure
how \emph{topologically} complicated primitive conjugacy classes are with respect to each other:

\begin{definition}[Co-surface graph $\CS$] \label{def:cosurface}
The \define{co-surface graph} $\CS$ of the free group $\free$ is the simplicial graph whose vertices are conjugacy classes of primitive elements, and where two vertices $\alpha$ and $\beta$ are joined by an edge if there is a once-punctured surface $S$ and an isomorphism $\pi_1(S) \cong \free$ with respect to which $\alpha$ and $\beta$ may both be represented by \emph{simple} closed curves on $S$.
\end{definition}

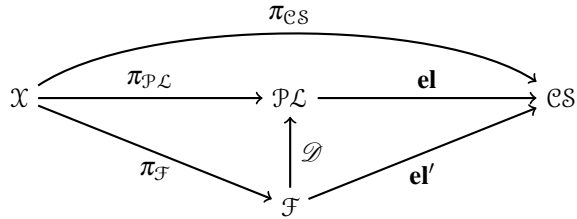
\begin{figure}[htbp]
\begin{center}
\begin{tikzpicture}[scale = 1.8]
\node (os) at (-2,0) {$\os$};
\node (cs) at (2,0) {$\cs$};
\node (pl) at (0,0) {$\pl$};
\node (fc) at (0,-.8) {$\fc$};
\draw[thick,->] (os) .. controls (-1.1,.6) and (1.1,.6) .. node[above] {$\pi_\cs$} (cs);
\draw[thick,->] (os) -- node[above] {$\pi_\pl$} (pl);
\draw[thick,->] (os) -- node[below] {$\pi_\fc$} (fc);
\draw[thick,->] (pl) -- node[above] {$\epl$} (cs);
\draw[thick,->] (fc) -- node[below] {$\efc$} (cs);
\draw[thick,->] (fc) -- node[right] {$\mathcal{D}$} (pl);
\end{tikzpicture}
\caption{The co-surface graph and its friends. All rightward pointing arrows are coarsely Lipschitz, and $\mathcal{D}$ is a quasi-isometry. Coarsely speaking, all maps are $\Out(\free)$--equivariant and the diagram commutes.}
\label{Fig:graphs}
\end{center}
\end{figure}


In other words, each once-punctured surface $S$ with $\pi_1(S)\cong \free$ determines a subset $C_S\subset\pl^0$ consisting of those primitive conjugacy classes that correspond to (nonseparating) simple closed curves on $S$. As our graph $\cs$ is obtained by collapsing each $C_S\subset\pl$ to a set of diameter $1$, it  
records the geometry of the primitive conjugacy classes that remains after all the ``surface sets'' $C_S$ have been crushed---hence the name ``co-surface'' graph.
 We equip $\cs$ with the path metric $d_\cs$ in which each edge has length $1$ and note that $\Out(\free)$ acts simplicially (and hence isometrically) on $\cs$.

From basic topology, we know that if primitive elements $a,b\in \free$ are jointly part of a free basis of $\free$, then one may build a once-punctured surface $S$ and an isomorphism $\pi_1(S)\cong \free$ under which  $a$ and $b$ correspond to disjoint simple closed curves on $S$. Therefore the ``identity'' map $\pl^0\to\cs^0$ extends to a simplicial and hence $1$--Lipschitz $\Out(\free)$--equivariant ``electrification'' map $\epl\colon\pl\to \cs$. Define $\efc\colon \fc\to \cs$ to be the composition $\efc = \epl\circ \mathcal{D}$. The purpose of this section is to establish the following essential properties of $\cs$. For the statement, recall that $\phi\in \Out(\free)$ is \define{fully irreducible} if no positive power of $\phi$ fixes any conjugacy class of proper free factors of $\free$.

\begin{theorem}[Properties of $\CS$] \label{th:cs_properties}
For the free group $\free$ of rank at least $3$, the co-surface graph $\CS$ is hyperbolic and the map $\epl\colon \pl \to \CS$ (and thus also $\efc\colon \fc\to \cs$) is Lipschitz and alignment preserving. Moreover, $\phi \in \Out(\free)$ acts as a loxodromic isometry of $\CS$ if and only if $\phi$ is atoroidal and fully irreducible.
\end{theorem}

To establish these properties, we show that $\CS$ is quasi-isometric to another $\Out(\free)$--graph that has appeared in several different forms in the literature under the name ``intersection graph'' (see \Cref{remark:history}). To define this, say that a conjugacy class of $\free$ is \define{geometric} if it is either primitive or it corresponds to the cusp of a once-punctured surface whose fundamental group is identified with $\free$. Define the \define{intersection graph} to be the bipartite graph $\int$ whose vertices are geometric conjugacy classes and very small, simplicial, nonfree trees, and where a conjugacy class $\alpha$ is joined by an edge to a tree $T$ if and only if $\ell_T(\alpha) = 0$. Note that there is an obvious action $\Out(\free) \curvearrowright \int$ and that the inclusion $\pl^0\to \int^0$ extends to an $\Out(\free)$--equivariant map $\pl\to \int$. Brian Mann and Patrick Reynolds have proven the following:

\begin{theorem}[Mann \cite{Mann-thesis}, Mann--Reynolds \cite{MR2}]
\label{thm:int_graph}
The intersection graph is hyperbolic and the natural map $\pl\to \int$ is Lipschitz and alignment preserving. 
\end{theorem}

The proof of these results can be found in Mann's thesis \cite{Mann-thesis}, which is publicly available through ProQuest. 
The main point of his proof is to show that the map $\pl\to \int$ fits the parameters of a recent theorem of Kapovich and Rafi \cite[Proposition 2.5]{KapRaf}. 
Applying the Kapovich--Rafi result shows that hyperbolicity of the primitive loop graph implies hyperbolicity the intersection graph and, moreover, that the map $\pl\to\int$ is alignment preserving.
One could use the same argument from \cite{Mann-thesis} to directly show that the Kapovich--Rafi result applies to the map $\epl\colon \pl\to \cs$; rather than carry out this argument in detail, we simply invoke \Cref{thm:int_graph} and the following quasi-isometry between the co-surface graph and the intersection graph:

\begin{proposition} \label{prop:qi_graphs}
The graphs $\int$ and $\CS$ are $\Out(\free)$--equivariantly quasi-isometric.
\end{proposition}
\begin{proof}
In \cite[\S2.4]{Mann-thesis}, Mann defined a graph $\mathcal{P}$ which he shows is quasi-isometric to $\int$. The vertices of $\mathcal{P}$ are marked $\rank(\free)$--petal roses, and roses $R$ and $R'$ are joined by an edge if they have either a common petal or a common ``cusp'' (meaning that the given isomorphism $\pi_1(R)\cong \pi_1(R')$ may be realized by $\pi_1$--isomorphically embedding $R$ and $R'$ into the same once-punctured surface).
We show that $\mathcal{P}$ and $\CS$ are quasi--isometric.

Define an equivariant map $\Phi\colon \mathcal{P} \to \CS$ by sending a rose $R$ to the conjugacy class determined by one of its petals (the set of possible choices has diameter $1$ by definition). Whenever vertices $R$ and $R'$ of $\mathcal{P}$ have a common petal we then have $d_{\cs}(\Phi(R),\Phi(R'))\le 2$, and whenever $R$ and $R'$ have a common cusp we have $d_{\cs}(\Phi(R),\Phi(R'))\le 1$ since in this case the petals of $R$ and $R'$ all correspond to simple closed curves on the same once-punctured surface $S$. Thus $\Phi\colon \mathcal{P}\to \cs$ is $2$--Lipschitz. Similarly, let $\Psi\colon \CS \to \mathcal{P}$ be an equivariant map that sends each primitive conjugacy class $\alpha$ to any rose in which $\alpha$ appears as a petal. Since the set of such roses has diameter $1$, $\Psi$ is a coarse inverse for $\Phi$. Hence, it suffices to show that $\Psi$ is Lipschitz. For this, if $\alpha$ and $\beta$ are adjacent vertices of $\cs$ we may choose a once-punctured surface $S$ in which $\alpha$ and $\beta$ are represented simple closed curves. Extending these curves to $\pi_1$--injectively embedded roses $R$ and $R'$ in $S$, we see that $d_{\mathcal{P}}(\Psi(\alpha),\Psi(\beta))\le 3$ since $R$ and $R'$ have a common cusp and are respectively adjacent to $\Psi(\alpha)$ and $\Psi(\beta)$ by construction. This shows that $\Psi$ is $3$--Lipschitz and completes the proof.
\end{proof}

\begin{proof}[Proof of \Cref{th:cs_properties}]
By \Cref{thm:int_graph}, $\int$ is hyperbolic and the map $\pl \to \int$ is Lipschitz and coarsely alignment preserving. Since $\int$ and $\CS$ are $\Out(\free)$--equivariantly quasi-isometric by \Cref{prop:qi_graphs}, the same holds for the map $\epl\colon \pl \to \CS$.
The classification of the loxodromic elements of the action $\Out(\free) \curvearrowright \cs$ is proven in \Cref{cor:lox_cs} below.
\end{proof}

We also note that Bestvina and Feighn show the projection $\plproj\colon \os\to \pl$ is alignment preserving \cite[Theorem 9.3]{BFhyp}. By \Cref{lem:generalized_alignment} it follows that the composition $\csproj\colonequals \e\circ\plproj\colon \os\to \cs$ is alignment preserving. As for $\fc$, for $G,H\in \os$ we then define
\[d_\cs(G,H) = \diam_{\cs}(\csproj(G)\cup \csproj(H)).\]

\begin{remark}[Historical context] \label{remark:history}
In \cite{kapovich2007geometric}, Kapovich and Lustig use their intersection form (c.f. \Cref{th:currents}) to show that several free group analogs of the curve complex have infinite diameter. Among their proposed graphs were (up to quasi-isometry) versions of the free factor complex, free splitting complex, and what they call the intersection graph. Their definition of the intersection graph is the following: vertices are \emph{all} conjugacy classes of very small trees and geodesic currents and a tree $T$ is joined by and edge to a current $\mu$ if $\langle T,\mu\rangle = 0$. Although this graph is not connected (e.g. a free simplicial tree is isolated), the connected component containing the rational currents corresponding to primitive conjugacy classes is $\Out(\free)$ invariant.  This version of the intersection graph however is different from the graph $\int$ defined above, which is also referred to as the intersection graph in \cite{Mann-thesis}. The difference between these graphs lies in exactly \emph{which} geodesic currents are allowed; different restrictions determine which electrification of $\F$ one obtains. Using $\CS$ avoids this ambiguity as well as having the added benefit of a natural and transparent definition. 
\end{remark}

\subsection{Boundary of $\CS$}
\label{sec:cosurfaceboundary}
From \Cref{th:boundaries,th:cs_properties} we deduce that $\partial \CS \cong \partial_{\CS} \F$. Our next lemmas show that $\partial_\CS \fc$ is precisely the collection of classes of free arational trees in $\partial \F$. The first lemma follows easily from work of Coulbois--Hilion--Reynolds and Bestvina--Reynolds.

\begin{lemma}\label{lem:free_arational}
Let $T \in \partial \X$ be free and arational and let $\mu$ be any geodesic current. If $\langle T,\mu\rangle = 0 = \langle S,\mu\rangle$ for some tree $S \in \X \cup \partial \X$, then $S$ is also free and arational.
\end{lemma}
\begin{proof}
By \Cref{th:supp_of_current}, the hypotheses imply $\supp(\mu)\subset L(T)$ and $\supp(\mu)\subset L(S)$.  Using that $T$ is free and arational, \Cref{th:CHR} moreover gives $\supp(\mu) = L'(T)$; thus $L'(T)\subset L(S)$. Now apply Proposition 4.2(i) and Corollary 4.3 of \cite{bestvina2012boundary} to conclude that $L(S) = L(T)$.

Recall from \S\ref{sec:trees} that a tree is free if and only if its dual lamination does not contain $L(\alpha)$ for any nontrivial conjugacy class $\alpha$. Since this necessarily holds for $L(T) = L(S)$, we conclude that $S$ is free as well. 
Finally, if $S$ were not arational then there would be a free factor $A$ of $\free$ and an $A$--invariant subtree $S^A\subset S$ on which $A$ acts with dense orbits. It would then follow that $L(T)=L(S)$ contains a leaf in $\partial^2 A \subset \partial^2 \free$. 
However, the fact that $T$ is free and arational implies that $L(T)$ cannot contain a leaf in $\partial^2 A$ \cite[Lemma 2.1]{Rey12}. This contradiction shows that $S$ must be arational and completes the proof of the lemma.
\end{proof}

The following lemma is an application of a standard argument for showing that graphs which are similar to the curve graph of a surface have infinite diameter. See \cite{Kob,MasurMinsky, bestvina2012boundary, kapovich2007geometric}. The details are provided for the reader's convenience. 

\begin{lemma} \label{lem:unbounded}
Let $(G_i)_{i\ge0}$ be a sequence of graphs in $\X$ converging
to a tree $T$ in $\partial \X$ which is free and arational. Then the projections $\pi_{\CS}(G_i)$ of $G_i$ to $\CS$ are unbounded.
\end{lemma}
\begin{proof}
Recall that $\partial \os$ is realized as the boundary of $\os$ inside the projective space $\mathbb{P}\R^\free$. Thus the statement $G_i\to T$ means that there is a sequence of scaling constants $\lambda_i>0$ such that $\langle \lambda_iG_i,c\rangle \to \langle T,c\rangle$ for every conjugacy class $c$ in $\free$. Observe that that $d_\os(G_0,G_i)$ is unbounded since $\fproj$ is Lipschitz and \Cref{th:factor_complex_boundary} implies that $\fproj(G_i)$ converges to the boundary of $\fc$. Therefore, there is a conjugacy class $c$ such that $\langle G_i, c\rangle$ is unbounded. If $\lambda_i \ge \epsilon > 0$ for all $i$, this would imply
\[\langle T,c\rangle = \lim_i \langle \lambda_i G_i ,c \rangle = \limsup_i \lambda_i \langle G_i, c\rangle \ge \epsilon \limsup_i \langle G_i, c\rangle = \infty,\]
contradicting that $\langle T, c \rangle =  \len_T(c) < \infty$. Therefore, after passing to a subsequence, we may assume $\lambda_i \to 0$.

Let $\alpha_i\in \cs^0$ be a primitive loop in the projection $\csproj(G_i)$; thus $\alpha_i$ corresponds to an embedded closed loop on $G_i$ and so $\langle G_i, \alpha_i\rangle \le 1$. Suppose that these curves do not go to infinity in $\CS$. Then, after passing to a subsequence and fixing some $x\in \cs^0$, we may assume that $d_{\cs}(x,\alpha_i) =M$ for all $i$. Build a geodesic $x=x_i^0,x_i^1,\ldots, x_i^M=\alpha_i$ for each $i\ge 0$. By definition of $\CS$, for each $0\le k \le M-1$ there is a once-punctured surface $S_i^k$ realizing the edge between $x_i^k$ and $x_i^{k+1}$; let $c_i^k$ be the conjugacy class corresponding to the cusp (i.e., peripheral curve) of $S_i^k$. Further, let $R_i^k$ be the simplicial tree dual to the simple closed curve representing $x_i^k$ on $S_i^k$, and for $1\le k \le M$, let $L_i^{k}$ be the simplicial tree dual to $x_i^k$ on $S_i^{k-1}$. By construction
\begin{align}\label{eq:zero}
\langle R_i^k, x_i^k\rangle = \langle R_i^k,c_i^k\rangle = 0 \qquad \text{and}\qquad \langle L_i^k, x_i^k\rangle= \langle L_i^k,c_i^{k-1}\rangle=0.
\end{align}
Now let $i \to \infty$ and, after passing to a subsequence, assume that everything converges projectively to either a tree or a geodesic current. Denote the limit by omitting the subscript. Since $\langle\cdot,\cdot\rangle$ is continuous and $G_i$ converges projectively to the free arational tree $T$, we necessarily have $\langle T,\alpha\rangle = 0$, where $\alpha$ is a projective limit of $\alpha_i=x_i^M$ in $\pcurr$. To see this, choose scaling constants $\mu_i> 0$ so that $\mu_i \alpha_i$ converges to the current $\alpha$. Letting $A>0$ denote the length of the shortest loop in $G_0$, we have 
\[\limsup_i \mu_i A \le \limsup_i \mu_i \langle G_0, \alpha_i\rangle = \lim_i \langle G_0, \mu_i\alpha_i\rangle \to \langle G_0,\alpha \rangle < \infty.\]
Using $\langle G_i,\alpha_i \rangle \le 1$, it follows that
\[\langle T,\alpha\rangle = \lim_i \langle \lambda_i G_i, \mu_i\alpha_i \rangle  = \lim_i \lambda_i\mu_i \langle G_i,\alpha_i\rangle \le \limsup_i \lambda_i \mu_i \le (\limsup_i \mu_i)(\lim_i \lambda_i) = 0,\]
as claimed. By continuity of $\langle \cdot,\cdot \rangle$, we additionally have $\langle L^M,\alpha \rangle = 0$. Hence, $L^M$ is free and arational by \Cref{lem:free_arational}. Similarly we have $\langle L^M, c^{M-1}\rangle = 0 = \langle R^{M-1},c^{M-1}\rangle$; thus $R^{M-1}$ is also free and arational by \Cref{lem:free_arational}. Using continuity and \eqref{eq:zero} again to pair $R^{M-1}$ and $L^{M-1}$ with $x^{M-1}$, we now see that $L^{M-1}$ is free and arational as well. 
Applying this augment inductively, we conclude that $R^0$ is free and arational. This, however, contradicts the observation that $\langle R^0,x^0\rangle =0$ for the primitive conjugacy class $x^0=x$ (recall that the sequence $x_i^0$ is constant). This shows that $d_\cs(x,\alpha_i) \to \infty$ as $i\to \infty$ and completes the proof.
\end{proof}

\begin{theorem}[Boundary of $\cs$]
\label{th:boundary_of_cs}\label{cor:boundary_cs}
The Gromov boundary $\partial \cs$ of the co-surface graph is $\Out(\free)$--equivariantly homeomorphic to the subspace of $\partial\fc$ consisting of classes of free arational trees. 
\end{theorem}

\begin{proof}
We use the alignment preserving map $\efc\colon \fc\to \cs$ and \Cref{th:boundaries} to identify $\partial \cs \cong \partial_\cs \fc$. By \Cref{th:factor_complex_boundary} and \Cref{lem:unbounded}, the set of free arational trees is contained in $\partial_{\CS}\F $. Further, if $T \in \partial \F$ then $T$ is arational by \Cref{th:factor_complex_boundary}. If $T$ is not free, then by \Cref{th:char_arational}, $T$ is dual to a measured lamination $L$ on a once--punctured surface $S$. Let $\alpha_i$ be a sequence of nonseparating simple closed curves in $S$ converging to the lamination $L$. Then $\alpha_i$ is also a sequence of rank $1$ free factors in $\F$ converging to $T \in \partial \F$ with $\diam_{\CS}\e(\alpha_i) \le 1$. Hence, $T \notin \partial_{\CS}\F$ and we conclude that the set of free arational trees equals $\partial_{\CS}\F$.
\end{proof}

\subsection{Quasi--isometric embeddings into $\CS$}
\label{sec:qi_into_cosurf}
We say that a finitely generated subgroup $\Gamma\le\Out(\free)$ \define{qi-embeds} into an $\Out(\free)$ graph $Y$ if some (equivalently any) orbit map $\Gamma\to Y$ is a quasi-isometric embedding. 
In this section, we prove that $\Gamma \le  \Out(\free)$ qi-embeds into $\CS$ if and only if it is purely atoroidal and qi-embeds into $\F$. This answers \cref{q:Kap} of I. Kapovich and clarifies the connection between the factor complex, the co-surface graph, and hyperbolic extensions of free groups. 

Fix a rose $R \in \X$ and a primitive conjugacy class $\alpha$ represented by a petal of $R$. Fix a finitely generated subgroup $\Gamma \le \Out(\free)$ such that the orbit map $\Gamma \to \F$ given by $g \mapsto g \cdot \alpha$ is a quasi-isometric embedding. In \cite{DT1}, we show that this implies that the orbit $\Gamma \cdot R$ has strong quasiconvexity properties in $\X$ (e.g. \Cref{th:DT_2}). For the application needed here, the following proposition from \cite{DKT} is most convenient.

\begin{proposition}[Folding rays to infinity {\cite[Proposition 5.6]{DKT}}]\label{lem:folding_rays}
Suppose that $\Gamma \le \Out(\free)$ is purely atoroidal and qi-embeds into $\fc$.
For any $k \ge0$ there is a $K\ge0$ such that if $(g_i)_{i\ge0}$ is a $k$--quasigeodesic ray in $\Gamma$, then there is an infinite length folding ray $\gamma\colon \I \to \X$ parameterized at unit speed with the following properties:
\begin{enumerate}
\item The sets $\gamma(\I)$ and $\{g_iR : i\ge 0\}$ have symmetric Hausdorff distance at most $K$.
\item The rescaled folding path $G_t = e^{-t}\cdot \gamma(t)\in\cv$ converges to the arational tree $T \in \partial\cv$ with the property that $\lim_{i\to \infty}g_i \cdot  \alpha= \partial \fproj(T)$ in $\F \cup \partial \F$, where $\partial \fproj(T)$ is the projection of the projective class of $T$ to the boundary of $\F$ (c.f. \Cref{th:factor_complex_boundary}). Moreover, the action $\free \curvearrowright T$ is free.
\end{enumerate}
\end{proposition}

Given \Cref{lem:folding_rays}, we show that the orbit map from $\Gamma$ to the co-surface graph is metrically proper.

\begin{proposition}\label{prop:proper_orbit}
Suppose that $\Gamma \le \Out(\free)$ is purely atoroidal and qi-embeds into the factor complex $\F$. Then for every $D \ge0$ there is an $N \ge 0$ so that 
\[ d_\CS(R,g\cdot R) \ge D\]
for all $g\in \Gamma$ with $|g| \ge N$.
\end{proposition}

\begin{proof}
Suppose not. Then there is a $D\ge0$ and a sequence $h_i \in \Gamma$ with $d_\cs(R,h_i\cdot R)\le D$ and $\abs{h_i} \to \infty$ as $i\to \infty$. After passing to a subsequence, we may assume that $h_i \to q \in \partial \Gamma$. Since $\Gamma \to \F$ is a quasi-isometric embedding, there is a unique $\lambda \in \partial \F$ such that $h_i \cdot \alpha \to \lambda$ in $\F \cup \partial \F$. (Recall that $\alpha \in \pi_\F(R)$.)
\begin{claim}\label{cl:replace}
The sequence $(h_i)_{i\ge0}$ can be replaced by a geodesic $(g_i)_{ i\ge0}$ in $\Gamma$ such that $g_i \to q \in \partial \Gamma$ and $d_\CS(\alpha, g_i \cdot \alpha) \le \overline D$, where $\overline D \ge 0$ depends only on the constant $D$ and the orbit map $\Gamma \to \F$.

\end{claim}
\begin{proof}[Proof of claim]
Let $(g_i)_{i\ge 0}$ be any geodesic sequence in $\Gamma$ with $g_0 =1$ and  $g_i \to q$ in $\Gamma \cup \partial \Gamma$. Then for each $i\ge0$ there is a $j \ge0$ such that any geodesic $[0,h_j]$ passes within $2\delta$ from $g_i$; thus the triple $(1,g_i,h_j)$ is $4\delta$--aligned in $\Gamma$. Since $\Gamma\to\fc$ is a quasi-isometric embedding and $\fc$ is hyperbolic, the stability of quasigeodesics (\Cref{prop:stability_of_quasis}) implies that $\Gamma\to\fc$ is alignment preserving. Therefore $\Gamma\to\cs$ is alignment preserving by \Cref{lem:generalized_alignment} and so there is some $C\ge 0$ (depending only on the quasi-isometry constant of $\Gamma\to\fc$) so that
\[d_\cs(\alpha,g_i\cdot \alpha) \le d_\cs(\alpha, g_i \cdot \alpha) + d_\cs(g_i\cdot \alpha, h_j\cdot \alpha) \le d_\cs(\alpha, h_j\cdot \alpha) + C \le D + C \equalscolon \overline{D}.\qedhere\]
\end{proof}

Now let $(g_i)_{i\ge0}$ be as in \Cref{cl:replace} and apply \Cref{lem:folding_rays} to obtain a folding ray $\gamma \colon \I \to \X$. By \Cref{lem:folding_rays}, the graphs $g_i \cdot R$ and $\gamma(t)$ both limit to the same free arational tree $T \in \partial \X$. Since $g_i \cdot R$ converges to $T$ in $\overline \X$, \Cref{lem:unbounded} implies that $d_{\CS}(R, g_i \cdot R) \to \infty$, contradicting the construction of $(g_i)_{i\ge 0}$.
\end{proof}

\begin{theorem}[Qi-embedding into $\cs$]
\label{th:qi_into}\label{th:qi} 
Let $\Gamma$ be a finitely generated subgroup of $\Out(\free)$. Then $\Gamma$ qi-embeds into the co-surface graph $\cs$ if and only if $\Gamma$ is purely atoroidal and convex cocompact.
\end{theorem}

\begin{proof}
First, if $\Gamma$ qi-embeds into $\CS$, then any orbit map $\Gamma \to \F$ is a quasi-isometric embedding since the $\Out(\free)$--equivariant map $\efc \colon \F \to \CS$ is Lipschitz. Let $g \in \Gamma$ be infinite order. Since no power of $g$ fixes a vertex of $\F$, $g$ is fully irreducible. If $g$ were not atoroidal, then Bestvina--Handel \cite{BH92} proved that $g$ is induced by a pseudo--Anosov homeomorphism on a once-punctured surface. However, if this were the case, then $g$ would have bounded orbits in $\CS$, a contradiction.

Now suppose that $\Gamma$ is purely atoroidal and that the orbit map $\Gamma\to\fc$ given by $g\mapsto g\cdot \alpha$, for some rank $1$ free factor $\alpha$, is a quasi-isometric embedding.
Since $\efc \colon \F \to \CS$ is alignment preserving, it follows that the map $g\mapsto g\cdot \efc(\alpha)\in \cs$ is alignment preserving. \Cref{prop:proper_orbit} shows that this orbit map $\Gamma\to \cs$ is also metrically proper. We thus conclude that $\Gamma\to \cs$ is a quasi-isometric embedding by \Cref{lem:CAP}
\end{proof}

\begin{corollary}\label{cor:lox_cs}
An element of $\Out(\free)$ acts loxodromically on $\cs$ if and only if it is fully irreducible and atoroidal.
\end{corollary}

\begin{proof}
An element $g$ of $\Out(\free)$ acts loxodromically on $\cs$ if and only if $\langle g \rangle$ qi-embeds into $\cs$. By \Cref{th:qi}, this is the case if and only if $g$ is atoroidal and acts loxodromically on $\F$. Since the loxodromic elements for the action $\Out(\free) \curvearrowright \F$ are precisely the fully irreducible elements \cite[Theorem 9.3]{BFhyp}, the corollary follows. 
\end{proof}

\section{Lifting to covers}
In this section we show that orbits in the co-surface graph are well-behaved when passing to finite index subgroups of $\free$. Specifically, we show that if $\Gamma \le \Out(\free)$ qi-embeds into $\cs$, then for any finite index $H \le \free$, the induced subgroup $\Gamma^H$ of $\Out(H)$ qi-embeds into the co-surface graph of $H$. 
This proposition will be necessary in \S\ref{sec:undistortion_in_fibers}, but it may also be of independent interest.

Fix a finite index subgroup $H$ of $\free$ and let $\Gamma \le \Out(\free)$ be finitely generated. Denote by $\Gamma_H$ the subgroup of $\Gamma$ consisting of outer automorphisms which fix the conjugacy class of $H$, and let $\Gamma^H$ be the induced subgroup of $\Out(H)$. That is, $f \in \Gamma^H$ if there is an automorphism $\phi$ of $\free$ whose outer class is in $\Gamma$ such that $\phi|_H$ is in the outer class $f$. These groups fit into a short exact sequence:

\begin{lemma} \label{lem:easy}
For any $\Gamma \le \Out(\free)$ and $H \le \free$ finite index, there is a short exact sequence 
\[
1 \longrightarrow N(H)/H \longrightarrow \Gamma^H \longrightarrow \Gamma_H \longrightarrow 1,
\]
where $N(H)\le \free$ is the normalizer of $H$ and $N(H)/H$ is finite. 
\end{lemma}
\begin{proof}
Let $\Aut_H(\free)$ be the subgroup of $\Aut(\free)$ fixing $H$, and let $\Aut^\free(H)$ be the image of $\Aut_H(\free)$ under the restriction map $\mathrm{res}\colon \Aut_H(\free)\to \Aut(H)$. Let $\Out_H(\free)$ and $\Out^\free(H)$ denote the images of these subgroups under the projection homomorphisms $p_\free\colon\Aut(\free)\to \Out(\free)$ and $p_H\colon \Aut(H)\to \Out(H)$, respectively. These fit into the following diagram:
\begin{center}
\begin{tikzpicture}[scale = 1.8]
\node (autFfixH) at (3,1) {$\Aut_H(\free)$};
\node (autHimageF) at (0,1) {$\Aut^\free(H)$};
\node (outFfixH) at (3,0) {$\Out_H(\free)$};
\node (outHimageF) at (0,0) {$\Out^\free(H)$};
\node (one) at (4,0) {$1$};
\node (kernel) at (-1.5,0) {$N(H)/H$};
\node (otherone) at (-2.5,0) {$1$};
\draw[thick,left hook->>] (autFfixH) -- node[above] {\small$\mathrm{res}$} (autHimageF);
\draw[thick,->>] (autFfixH) -- node[right] {\small$p_\free$} (outFfixH);
\draw[thick,->>] (autHimageF) -- node[right] {\small$p_H$} (outHimageF);
\draw[thick,->] (autHimageF) -- node[above,sloped] {\small$\hat\pi=p_\free\circ\mathrm{res}\inv$} (outFfixH);
\draw[thick,->] (outHimageF) -- node[above] {\small$\pi$} (outFfixH);
\draw[thick,->] (otherone) -- (kernel);
\draw[thick,->] (kernel) -- (outHimageF);
\draw[thick,->] (outFfixH) -- (one);
\end{tikzpicture}
\end{center}

The restriction map $\mathrm{res}\colon \Aut_H(\free)\to \Aut^\free(H)$ is surjective by definition of $\Aut^\free(H)$. It is also injectve. To see this, suppose $\phi_1,\phi_2\in \Aut_H(\free)$ are automorphisms fixing $H$ that have the same restriction to $H$, and let $a\in \free$ be arbitrary. We may choose $k > 0$ so that $a^k\in H$ and consequently $\phi_1(a)^k = \phi_2(a)^k$. As elements of $\free$ have unique roots, it follows that $\phi_1(a) = \phi_2(a)$, proving the claim.

Identify $\free$ with is image in $\Aut(\free)$ under the map sending each $a\in \free$ to the conjugation automorphism $g\mapsto aga\inv$. With respect to this identificaion, the normalizer of $H$ in $\free$ is exactly 
\[N(H) = \free\cap \Aut_H(\free) = \ker(p_\free\colon \Aut_H(\free)\to\Out_H(\free)).\]
 Thus the kernel of the composition $\hat\pi\colonequals p_\free\circ\mathrm{res}\inv \colon \Aut^\free(H)\to \Out_H(\free)$ may be naturally identified with $N(H)$. 
Observe that  every inner automorphism of $H$ lifts to an inner automophism of $\free$ via the map $\mathrm{res}\inv$. Since the resulting inner automorphism of $\free$ clearly normalizes $H$, we furthermore see that $H\cong \Inn(H) = \ker(p_H)$ is contained in $N(H) = \ker(\hat\pi)$. It follows that the homomorphism $\hat\pi$ descends to a well-defined epimorphism $\pi\colon \Out^\free(H)\to \Out_H(\free)$ with kernel $\ker(\pi) = \ker(\hat{\pi})/\ker(p_H) \cong N(H)/N$.
The lemma now follows by observing that for any subgroup $\Gamma \le \Out(\free)$, we have
$\Gamma_H = \Gamma \cap \Out_H(\free)$ and $\Gamma^H = \pi^{-1}(\Gamma_H)$. 
\end{proof}


The main result of this section is the following:
\begin{proposition} 
\label{prop:fi_lift}\label{prop:lifting_qi}
Let $H$ be a finite index subgroup of $\free$ and let $\Gamma^H$ denote the subgroup of $\Out(H)$ induced by elements of $\Gamma\le\Out(\free)$ that stabilize the conjugacy class of $H$. If $\Gamma$ is finitely generated and qi-embeds into $\cs$, then $\Gamma^H$ also qi-embeds into $\cs(H)$.
\end{proposition}

Let us briefly remark on the use of the co-surface graph in the statement of \Cref{prop:lifting_qi}. In particular, the corresponding statement for the factor graph $\F$ is false. For example, let $\phi \in \Out(\free)$ be an automorphism that  can be represented by a pseudo-Anosov on a once--punctured surface $S$. Let $H$ be a subgroup of $\free\cong \pi_1(S)$ corresponding to a cover $\tilde S\to S$ with at least $2$ punctures. The cyclic subgroup $\Gamma = \langle \phi \rangle$ then quasi-isometrically embeds into $\F(\free)$ since $\phi$ is fully irreducible. However, $\Gamma^H$ does not qi-embed into $\F(H)$. Indeed, $\Gamma^H$ is a virtually cyclic group whose infinite order elements are represented by lifts of powers of $\phi$; since each such lift permutes the punctures of $\tilde{S}$ and each puncture represents a rank $1$ free factor of $H$, $\Gamma^H$ has bounded orbits in $\F(H)$. This suggests that $\cs$ is a better tool for studying finite index subgroups of $\free$.

The proof of \Cref{prop:lifting_qi} requires the following result of Reynolds whose proof uses ideas of Guirardel.

\begin{lemma}[Reynolds {\cite[Lemma 4.2]{reynolds2011indecomposable}}]\label{lem:ired_dec}
Suppose that $G \curvearrowright T$ is an indecomposable action and that $H\le G$ is finitely generated and finite index. Then the action $H \curvearrowright T^H$ is indecomposable.
\end{lemma}

We will use the lemma in the following form. 
\begin{corollary}\label{cor:arationals_lift}
Suppose that $T \in \partial \X$ is free and arational and that $H \le \free$ is finite index. Then the minimal $H$--subtree $T^H$ is also free and arational. 
\end{corollary}
\begin{proof}
Clearly, $T^H$ is free because $T$ is free. Since $T$ is arational, it is indecomposable (\Cref{th:char_arational}); thus $T^H$ is also indecomposable by \Cref{lem:ired_dec}. Using \Cref{th:char_arational} again, we conclude that $T^H$ is arational.
\end{proof}

\subsection{The Outer space of a subgroup}
Fix $H \le \free$ a subgroup of finite index $[\free:H] = n$. Then $H$ is a free group of rank $1-n(1-\rank(\free))$ and we denote its Outer space by $\X(H)$. Recalling that $\rose$ is our fixed $\rank(\free)$--petal rose used to mark graphs of $\X$, we let $H\vert\rose$ denote the corresponding $H$--cover and fix a homotopy equivalence between $H\vert\rose$ and a rose $\rose_H$. 

There is a natural inclusion $i^* \colon \X \to \X(H)$ defined by taking $H$--covers and lifting markings. In details, if $\phi \colon \rose \to G$ is a marked metric $\free$--graph, then the $H$--cover $H\vert G$ is a metric $H$--graph of volume $n$ and we may choose a lift $\phi_H\colon H\vert\rose\to H\vert G$. Any other lift $\tilde{\phi} \colon H\vert \rose \to H\vert G$ is then obtained by precomposing $\phi_H$ with an element of the deck group of $H|\rose \to \rose$, which is isomorphic to $N(H) /H$. Since for each such deck transformation $d \in N(H)/H$, there is a graph isometry $\rho_d$ of $H\vert G$ such that $\rho_d \circ \phi_H \sim \phi_H \circ d$, we see that the equivalence class of $(H|G, \phi_H)$ in $\X(H)$ is well-defined. Using the homotopy lifting property, we additionally see that this induces a well-defined map 
\[
i^* \colon \X \to \X(H)\]
given by $[G,\phi] \mapsto [H|G,\phi_H]$.
Also note that if $\tilde g \in \Gamma^H$ projects to $g \in \Gamma_H$, then $\tilde{g} \cdot i^*(G) = i^*(g\cdot G)$. Even better, this map is an isometry with respect to the Lipschitz metric:
\begin{proposition}\label{pro:Outer_subgroup}
Let $H \le \free$ be a finite index subgroup. Then the induced map $i^* \colon \X \to \X(H)$ is an isometry with respect to the Lipschitz metric. Moreover, $i^*$ maps folding paths in $\X$ to folding paths in $\X(H)$.
\end{proposition}
\begin{proof}
Fix $G_1$ and $G_2$ in $\X$. If $f \colon G_1 \to G_2$ is an optimal change of marking, then choosing a lift $f_H \colon H|G_1 \to H|G_2$ we see that since $\mathrm{Lip}(f_H) = \mathrm{Lip}(f)$, 
\[
d_{\X(H)}(i^*(G_1),i^*( G_2)) = d_{\X(H)} (H|G_1,H|G_2) \le d_\X(G_1,G_2).
\] 
Further, 
\[
e^{d_{\X(H)} (H|G_1,H|G_2)}  = \sup_{h \in H}\frac{\ell(h\vert (H\vert G_2))}{\ell(h\vert (H\vert G_1))} = \sup_{h\in H}\frac{\ell(h|G_2)}{\ell(h|G_1)}.
\]
Let $a \in \free$ be such that $a$ is optimally stretched by $f\colon G_1 \to G_2$, i.e. $\ell(a|G_2) = \mathrm{Lip}(f)\ell(a|G_1)$, and let $k >0$ be the smallest positive integer such that $a^k \in H$. Then using our observation above
\[
e^{d_{\X} (G_1,G_2)}  = \frac{\ell(a|G_2)}{\ell(a|G_1)} = \frac{\ell(a^k|G_2)}{\ell(a^k|G_1)} \le \sup_{h\in H}\frac{\ell(h|G_2)}{\ell(h|G_1)} =  e^{d_{\X(H)} (H|G_1, H|G_2)} .
\]
Hence, we also have $d_\X(G_1,G_2) \le d_{\X(H)}(i^*(G_1),i^*( G_2))$ showing that $i^* \colon \X \to \X(H)$ is an isometry.

To show that $i^*$ maps folding paths to folding paths, let $\{G_t\}_{t\in [0,L]}$ be a folding path in $\os$ with corresponding folding maps $\{\phi_{st}\colon G_s\to G_t\}_{s<t}$ (see \S\ref{sec:folding} for the definitions of folding maps and folding paths). Set $\tilde{G}_t = H\vert G_t = i^*(G_t)$ with covering map $p_t\colon \tilde{G}_t\to G_t$. We claim that $\{\tilde{G}_t\}$ is a folding path in $\os(H)$ with corresponding lifted folding maps $\{\tilde{\phi}_{st}\colon H\vert G_s \to H\vert G_t\}_{s<t}$. Indeed, that each of these lifts is a local $\Lip(\tilde{\phi}_{st})$--homothety with $\Lip(\tilde{\phi}_{st})=d_{\os(H)}(\tilde{G}_s,\tilde{G}_t)$ follows from what we have already shown. What's more, the equality $p_t\circ\tilde{\phi}_{st}= \phi_{st}\circ p_s$ implies that the illegal turn structure that $\tilde{\phi}_{st}$ induces on $\tilde{G}_s$ is exactly the lift (via $p_s$) of the illegal turn structure that $\phi_{st}$ induces on $G_s$. From this it follows that each $\tilde{\phi}_{st}$ is a folding map and, by the uniqueness of folding paths \cite[\S2]{BFhyp}, that $\{\tilde{G}_t\}$ is the folding path determined by $\tilde{\phi}_{0L}\colon \tilde{G}_0\to \tilde{G}_L$.
\end{proof}

\subsection{Proof of \Cref{prop:lifting_qi}}
Combining our work in the previous sections, we now turn to the proof of \Cref{prop:lifting_qi}. Let $R \in \X$ be a rose with a petal representing $\alpha \in \CS^0$. Let $\tilde \alpha \in \CS^0(H)$ be primitive conjugacy class of $H$ represented by an embedded loop of $H|R$ covering this petal of $R$. Note that $\alpha \in \pi_\CS(R)$ and $\tilde \alpha \in \pi_{\CS}(H|R)$, where $\pi_\CS$ is the projection from outer space to the corresponding co-surface graph. 

Fix a finite generating set $\tilde S$ of $\Gamma^H$ with projection $S\subset \Gamma_H$. By abuse of notation, we identify $\Gamma^H$ and $\Gamma_H$ with the corresponding Cayley graphs $\cay{S}{\Gamma^H}$ and $\cay{\bar{S}}{\Gamma_H}$, which are geodesic metric spaces. The orbits maps $\tilde g\mapsto \tilde g\cdot \tilde \alpha$ and $\tilde g\mapsto \tilde g\cdot H\vert R$ then extend equivariantly to continuous maps $\psi\colon \Gamma^H\to \cs(H)$ and $\Psi\colon \Gamma^H\to \os(H)$, and similarly for $\Gamma_H$. Note also that the projection $\Gamma^H\to \Gamma_H$ is a quasi-isometry.

\begin{proof}[Proof of \Cref{prop:lifting_qi}]
By \Cref{lem:CAP}, it suffices to show that the orbit map $\psi\colon \Gamma^H \to \CS(H)$ is alignment preserving and metrically proper. The strategy is to relate the corresponding orbit in $\X$ with folding paths via \Cref{th:DT_2} and to use that these folding paths lift to $\X(H)$ by \Cref{pro:Outer_subgroup}. Note that since $\Gamma$ qi-embeds into $\CS$, $\Gamma$ is hyperbolic. Further, as $\Gamma_H$ is finite index in $\Gamma$, $\Gamma_H$, and hence $\Gamma^H$, is also hyperbolic.

In details, let $(\tilde g_-,\tilde g_0,\tilde g_+)$ be an aligned triple in $\Gamma^H$ and let $\tilde\gamma\colon \I\to \Gamma^H$ be a geodesic passing through $\tilde g_0$ with $\tilde\gamma(\I_\pm) = \tilde g_\pm$. The projection $\gamma\colon \I\to \Gamma_H\le \Gamma$ is then a uniform quasigeodesic and so maps to a uniform quasigeodesic in $\cs$ by assumption.  Hence \Cref{th:DT_2} provides a folding path $\{G_t\}_{t\in[0,L]}$ in $\X$ that has uniformly bounded symmetric Hausdorff distance from the image of the composition $\I\to \Gamma_H\to \os$. Since $i^*\colon \os\to \os(H)$ is a $\Gamma^H$--equivariant  isometric embedding (\Cref{pro:Outer_subgroup}), the same holds for the image of $\I\to \Gamma^H\to\os(H)$ and the folding path $i^*(G_t) = H\vert G_t$. In particular, the three points $\tilde g_-\cdot H\vert R$, $\tilde g_0\cdot H\vert R$, and $\tilde{g}_+\cdot H\vert R$ all lie within uniformly bounded symmetric Hausdorff distance of the geodesic $\{H\vert G_t\}$ in $\os(H)$. Since the projection $\csproj\colon \os(H)\to \cs(H)$ is alignment preserving, it follows that the triple $(\tilde{g}_-\cdot \tilde\alpha,\tilde{g}_0\cdot\tilde\alpha,\tilde{g}_+\cdot \tilde\alpha)$ is uniformly aligned in $\cs(H)$. Therefore the orbit map $\Gamma^H\to \cs(H)$ is alignment preserving.

It remains to show that $\Gamma^H \to \CS(H)$ is metrically proper. If this were not the case, then there is a sequence of $x_i \in \Gamma^H$ with $\abs{x_i} \to \infty$ but $d_{\CS(H)}(\tilde \alpha, x_i \cdot \tilde \alpha) \le D$ for some $D \ge 0$. By compactness of $\Gamma^H \cup \partial \Gamma^H$, after passing to a subsequence we may assume that $x_i \to q \in \partial \Gamma^H = \partial \Gamma_H = \partial \Gamma$. Since we have already shown that $\Gamma^H \to \CS(H)$ is coarsely alignment preserving, use \Cref{cl:replace} to obtain a geodesic $(\tilde g_i)_{i\ge0}$ in $\Gamma^H$ such that $\tilde g_i \to q$ in $\Gamma^H \cup \partial \Gamma^H$ but $d_{\CS(H)}(\tilde{\alpha}, \tilde{g}_i \cdot \tilde \alpha)$ is uniformly bounded. 

Now pass to a subsequence such that $g_i \cdot R$ converges to $T$ in $\overline \X$. Since $(g_i \cdot \alpha)_{i\ge0}$ is a quasigeodesic in $\CS$, we see combining \Cref{th:factor_complex_boundary} and \Cref{cor:boundary_cs} that the tree $T$ is free and arational and that $g_i \cdot \alpha$ converges to the equivalence class of $T$ in $\CS \cup \partial \CS$. Hence, $(H|g_i \cdot R = \tilde g_i \cdot H|R)_{i\ge0}$ converges in $\overline \X(H)$ to the tree $T^H$. By \Cref{cor:arationals_lift}, $T^H$ is free and arational. Hence, by \Cref{lem:unbounded}, $\pi_{\CS(H)}(\tilde g_i \cdot H|R) = \tilde g_i \cdot \tilde \alpha$
is unbounded in $\cs(H)$.
 In particular, $d_{\CS(H)}(\tilde{\alpha}, \tilde{g}_i \cdot \tilde \alpha) \to \infty$ contradicting the conclusion of the previous paragraph. This shows that $\Gamma^H \to \CS(H)$ is metrically proper and completes the proof.
 \end{proof}

\section{Flaring of simple conjugacy classes in Outer space orbits}
\label{sec:orbit_flaring}

In this section, we use the geometry of Outer space and the nature of folding paths to analyze how the lengths of conjugacy classes behave along the orbit $\Gamma\cdot R\subset \os$ of certain subgroups $\Gamma\le \Out(\free)$. When $\Gamma$ qi-embeds into the factor complex $\fc$, we will find that for simple conjugacy classes $\alpha\in \free$, the length $\len(\alpha\vert g\cdot R)$ grows roughly exponentially in the distance from a certain uniformly bounded-diameter subset $\minset{\Gamma}{R}(\alpha)\subset \Gamma$. Our analysis culminates in the rather technical \Cref{lem:flare_away_from_minimizer2} which establishes this exponential flaring not only for simple classes $\alpha$, but also for all conjugacy classes $\beta$ that are, in a sense, ``well-aligned'' with $\alpha$ at the points of $\minset{\Gamma}{R}(\alpha)\cdot R$ (see \Cref{def:almost_containment}). \Cref{lem:flare_away_from_minimizer2} moreover shows that this exponential growth is \emph{uniform} in all such $\alpha$ and $\beta$.

\subsection{Uniform bounded backtracking} 
\label{sec:bounded backtracking}

It is well known that any map $f\colon G\to H$ of metric core graphs has \define{bounded backtracking}, meaning that there is a constant $\mathrm{BBT}(f) \ge 0$ such that for any two points $p,q\in \tilde{G}$ in the universal cover and any lift $\tilde{f}\colon \tilde{G}\to \tilde{H}$ of $f$ one has that the path $\tilde{f}([p,q])$ is contained in the $\mathrm{BBT}(f)$--neighborhood of the geodesic segment $[\tilde{f}(p),\tilde{f}(q)]$; see, e.g., \cite{GJLLindex} or \cite{CHL2}. We will need a uniform bound on the constant $\mathrm{BBT}(f)$ over a broad family of graph maps. While bounds of this type are certainly well known to experts (see, e.g., \cite[Lemma 3.1]{BFHlam}), we include a short proof here for completeness.

\begin{lemma}[Backtracking bound]
\label{lem:backtracking bound}
For every $D > 0$ there exists a constant $C > 0$ so that if $G,H\in \os$ satisfy $\dsym(G,H)\le D$, then there exists a change of marking map $\phi\colon G\to H$ with $\mathrm{BBT}(\phi)\le C$.
\end{lemma}
\begin{proof}
The hypothesis ensures there are maps $\phi\colon G\to H$ and $\varphi\colon H\to G$ so that $\phi\circ \varphi\simeq \mathrm{Id}_H$, $\varphi\circ \phi\simeq \mathrm{Id}_G$, and $\Lip(\phi),\Lip(\varphi)\le e^D$. Set $K = e^D$. 
Choose any lifts $\tilde{\phi}\colon \tilde{G}\to\tilde{H}$ and $\tilde{\varphi}_0\colon \tilde{H}\to\tilde{G}$ of these maps to the universal covers, and note that $\tilde{\phi}$ and $\tilde{\varphi}_0$ are both equivariant with respect to the $\free$ actions on the trees $\tilde{G}$ and $\tilde{H}$. Note also that $\Lip(\tilde{\phi}),\Lip(\tilde{\varphi}_0)\le K$.

Choose a basepoint $v\in \tilde{G}$ and let $w=\tilde{\phi}(v)\in \tilde{H}$. Since the orbit $\free\cdot v$ is $1$--dense in $\tilde{G}$, we may choose $a \in \free$ so that $d_{\tilde{G}}(v,a \cdot\tilde{\varphi}_0(w))\le 1$. Define a new map $\tilde{\varphi}\colon\tilde{H}\to\tilde{G}$ by the rule $\tilde{\varphi}(q) = a \cdot\tilde{\varphi}_0(q)$ and note that we still have $\Lip(\tilde{\varphi})\le K$.
We now have that $d_{\tilde{G}}(v,\tilde{\varphi}\circ\tilde{\phi}(v)) \le 1$. By equivariance and the fact that $\tilde\varphi$ and $\tilde\phi$ are $K$--Lipschitz, it easily follows that $\tilde{\varphi}\circ\tilde{\phi}$ moves points at most distance $K^2+2$. 
From this, and the inequalities $\Lip(\tilde\varphi),\Lip(\tilde\phi)\le K$, one may conclude that $\tilde\phi$ is a $(2K^2+4)$--quasi-isometry.
Thus for any geodesic segment $[p,q]\subset\tilde{G}$, the image $\tilde{\phi}([p,q])$ is a $(2K^2+4)$--quasigeodesic in $\tilde{H}$. It now follows from \Cref{prop:stability_of_quasis} that $\tilde{\phi}([p,q])$ is contained in the $R_0(2e^{2D}+4,0)$--neighborhood of $[\tilde{\phi}(p),\tilde{\phi}(q)]$, as required.
\end{proof}

If $\varphi\colon T_0\to T$ is an $\free$--equivariant map between free simplicial trees, we also write $\mathrm{BBT}(\varphi)$ for the bounded backtracking constant of the induced map $T_0/\free\to T/\free$ of the quotient graphs.

\subsection{Bestvina--Feighn folding}
\label{sec:more_folding}
Here, we recall some additional facts about folding paths that we will need in the proof of \Cref{lem:flare_away_from_minimizer2}. Recalling the notation from \S\ref{sec:folding}, we see that the folding maps $\{\phi_{st}\colon G_s\to G_t\}_{s<t}$ associated to a folding path $\gamma\colon\I\to \os$ give rise to a well-defined illegal turn structure on each graph $G_s = \gamma(s)$ in the image (except for the right endpoint $\gamma(\Ipl)$ when $\Ipl < \infty$). We then say that an immersed path (i.e., segment) in $G_s$ is \define{legal} if it only takes legal turns;  notice that for any legal path $\beta\colon \J\to G_s$, the composition $\phi_{st}\circ\beta$ is a legal path in $G_t$. An immersed path in $G_s$ will be called \define{illegal} if it does not contain a legal subpath of length $3$.

If $G_t\in \os$ is equipped with an illegal turn structure, then for every conjugacy class $\alpha$ of $\free$ the immersed loop $\alpha\vert G_t\to G_t$ breaks into maximal legal segments separated by illegal turns.
Following our convention from \cite[\S6]{DT1}, the \define{legal length} of $\alpha\vert G_t$ is defined to be the sum $\leg(\alpha\vert G_t)$ of the lengths of those maximal legal segments that have length at least $3$. The following basic fact appears as Lemma 6.10 of \cite{DT1} and follows directly from work in \cite{BFhyp}:

\begin{lemma}[\cite{DT1}]
\label{lem:legal_flare}
For any folding path $G_t$, $t\in [a,b]$, every nontrivial conjugacy class $\alpha\in \free$ satisfies
\[\leg(\alpha\vert G_b) \ge \leg(\alpha\vert G_a) \left(\frac{1}{3}\right) e^{b-a}.\]
\end{lemma}

We also have the following technical result of Bestvina--Feighn; for the statement, recall that a nontrivial element $a \in \free$ (or the conjugacy class $\alpha$ thereof) is \define{simple} if $a$ is contained in a proper free factor of $\free$.

\begin{lemma}[{Bestvina--Feighn \cite[Lemma 5.8]{BFhyp}}]
\label{lem:illegal_flare}
There exists a constant $B_1$ depending only on $\rank(\free)$ with the following property. If $G_t$, $t\in [a,b]$ is a folding path and $\alpha$ is simple with $\alpha\vert G_t$ illegal for all $t\in [a,b]$, then either $\len(\alpha\vert G_a) > 2\len(\alpha\vert G_b)$ or else $d_\fc(G_a,G_b)\le B_1$.
\end{lemma}

For any subgroup $A\le \free$ (or conjugacy class $\alpha$) the illegal turn structure on $G_t$ pulls back to give an illegal turn structure on $A\vert G_t$ (or $\alpha\vert G_t$). This gives a notion of legal and illegal paths in $A\vert G_t$: an immersed path in $A\vert  G_t$ is (il)legal if and only if it is mapped to an (il)legal path in $G_t$. Define now the \define{illegality constant}
\[\illegality \colonequals (2\rank(\free)-1)(18\breve{m}(3\rank(\free)-3)+6),\]
where $\breve{m}$ is the maximal possible number of illegal turns in any illegal turn structure on any graph $G\in \os$ (so $\breve{m}$ is linear in $\rank(\free)$). In \cite{BFhyp}, Bestvina and Feighn introduced the following projections to folding paths.

\begin{definition}
\label{def:project_to_fold_path}
Given a folding path $\gamma\colon \I\to \os$ and a free factor $A$ of $\free$, set
\begin{align*}
\mathrm{left}_\gamma(A) =& \inf \{t\in \I : A\vert \gamma(t)\text{ has an immersed legal segment of length }3\}\quad\text{and} \\
\mathrm{right}_\gamma(A) =& \sup\{t\in \I : A\vert \gamma(t)\text{ has an immersed illegal segment of length }\illegality\}.
\end{align*}
We similarly define $\mathrm{left}_\gamma(\alpha)$ and $\mathrm{right}_\gamma(\alpha)$ for every nontrivial conjugacy class $\alpha$ of $\free$. Note that these definitions agree when $\alpha$ is a primitive conjugacy class.
\end{definition}

The following technical result played a key roll in Bestvina and Feighn's proof of \Cref{thm:fc_hyperbolic}.

\begin{proposition}[{Bestvina--Feighn \cite[Proposition 6.10]{BFhyp}}]
\label{prop:BF_strong_contract}
There exists a universal constant $B_2$ depending only on $\rank(\free)$ such that for every folding path $\gamma\colon \I\to \os$ and free factor $A$ of $\free$ we have
\[d_\fc\big(\gamma(\mathrm{left}_\gamma(A)),\gamma(\mathrm{right}_\gamma(A))\big)\le B_2.\]
\end{proposition}

We henceforth write $\fcbound$ for the universal constant $\max\{B_1,B_2\}$, where $B_1,B_2$ are the constants provided by \Cref{lem:illegal_flare} and \Cref{prop:BF_strong_contract}.

\subsection{Flaring and almost containment}
\label{sec:almost containment}

The following notion will help us relate distances in the metric graph bundle $\bund$ (see \S\ref{sec:metric_bundle}) to conjugacy lengths along folding paths.

\begin{definition}
\label{def:almost_containment}
Suppose that $\alpha,\beta$ are nontrivial conjugacy classes of $\free$ and that $G\in \os$ is a point in Outer space with corresponding $\R$--tree $\tilde{G}$. We say that \define{$\beta$ is $k$--almost contained in $\alpha$ at $G$}, for $k\ge 0$, if there exists an axis $\mathcal{A}\subset \tilde{G}$ for (an element of) $\alpha$ and a fundamental domain $B$ of an axis for (an element of) $\beta$ in $\tilde{G}$ so that $B\setminus \mathcal{A}$ is a (possibly degenerate) connected segment of length at most $k$. This is to say that $B$ is contained in $\mathcal{A}$ except for a subsegment of length at most $k$.
\end{definition}

Equivalently, $\beta$ is $k$--almost contained in $\alpha$ at $G$ if and only if there exist axes $\mathcal{A},\mathcal{B}\subset \tilde{G}$ for elements of the classes $\alpha$ and $\beta$, respectively, so that $\len(\mathcal{A}\cap\mathcal{B}) \ge \len(\beta\vert G) - k$. We also observe that $\alpha$ is always $0$--contained in itself.

\begin{lemma}
\label{lem:nearby containment}
For any $k\ge 0$ and $D\ge 0$ there exists a constant $k'\ge 0$ so that if $\beta$ is $k$--almost contained in $\alpha$ at $G\in \os$, then $\beta$ is $k'$--almost contained in $\alpha$ at $H$ for any $H\in \os$ with $\dsym(G,H)\le D$.
\end{lemma}
\begin{proof}
By assumption, we may may choose elements $a\in \alpha$ and $b\in \beta$ whose respective axes $\mathcal{A},\mathcal{B}\subset \tilde{G}$ satisfy $\len(\mathcal{A}\cap \mathcal{B})\ge \len(\beta\vert G) - k$. Orient $\mathcal{B}$ so that $b$ translates in the forward direction, and let $p$ and $q$ be the initial and terminal endpoints of the segment $\mathcal{A}\cap\mathcal{B}\subset\mathcal{B}$. Also let $x = b\cdot p\in \mathcal{B}$. By $k$--almost containment, we have that either $d_{\tilde{G}}(x,q)\le k$ or else $x\in \mathcal{A}\cap \mathcal{B}$. 

By \Cref{lem:backtracking bound}, there is an $e^D$--Lipschitz change of marking map $\varphi\colon G\to H$ with $\mathrm{BBT}(\varphi)\le C$ for some $C> 0$ depending only on $D$. Choose a lift $\tilde{\varphi}\colon\tilde{G}\to\tilde{H}$ and let $\mu$ be the (possibly empty) geodesic segment obtained by removing the length--$C$ initial and terminal segments of $[\tilde{\varphi}(p),\tilde{\varphi}(q)]$. Let $\mathcal{A}',\mathcal{B}'\subset \tilde{H}$ be the axes for $a\curvearrowright \tilde{H}$ and $b\curvearrowright \tilde{H}$. By definition of bounded backtracking, $\tilde{\varphi}(\mathcal{A})$ is contained in the $C$--neighborhood of $\mathcal{A}'$; in particular $\tilde{\varphi}(p)$ and $\tilde{\varphi}(q)$ are both within $C$ of $\mathcal{A'}$. It follows that $\mu\subset\mathcal{A}'$. Similarly $\mu\subset\mathcal{B}'$. Thus 
\[\len(\mathcal{A}'\cap\mathcal{B}')\ge \len(\mu) \ge d_{\tilde{H}}(\tilde{\varphi}(p),\tilde{\varphi}(q)) - 2C.\]
Now if $x= b\cdot p\in\mathcal{A}\cap\mathcal{B}$, then the same reasoning gives
\[2C+\len(\mathcal{A}'\cap\mathcal{B}') \ge d_{\tilde{H}}(\tilde{\varphi}(p),\tilde{\varphi}(x)) = d_{\tilde{H}}(\tilde{\varphi}(p),b\cdot \tilde{\varphi}(p)) \ge \len(\beta\vert H).\]
Otherwise we have $d_{\tilde{G}}(q,x)\le k$, so that
\begin{align*}
\len(\beta\vert H) &\le d_{\tilde{H}}(\tilde{\varphi}(p),b\cdot \tilde{\varphi}(p))
\le d_{\tilde{H}}(\tilde{\varphi}(p),\tilde{\varphi}(q)) +d_{\tilde{H}}(\tilde{\varphi}(q),\tilde{\varphi}(x))\\
&\le \big(\len(\mathcal{A}'\cap\mathcal{B}') + 2C\big) + e^Dk,
\end{align*}
since $\tilde{\varphi}$ is $e^D$--Lipschitz. Thus $\beta$ is $(2C+e^Dk)$--almost contained in $\alpha$ at $H$.
\end{proof}

We now make a simple observation. Recall from \Cref{def:project_to_fold_path} that $\mathrm{right}_\gamma(\alpha)$ denotes the supremum of times along a folding path $\gamma$ for which $\alpha\vert \gamma(t)$ contains an immersed illegal segment of length $\illegality$.
\begin{lemma}
\label{lem:legal_containment_flare}
Let $\gamma\colon \I\to\os$ be a folding path and suppose that $t\in \I$ satisfies $t\ge \mathrm{right}_\gamma(\alpha)$ for some conjugacy class $\alpha$ of $\free$. If $\beta$ is $k$--almost contained in $\alpha$ at $G_t=\gamma(t)$ and $\len(\beta\vert G_t)\ge 3k+3\illegality$, then
\[\leg(\beta\vert G_t) \ge \frac{2}{\illegality}\len(\beta\vert G_t).\]
\end{lemma}
\begin{proof}
The loop $\beta\vert G_t$ subdivides into two immersed subsegments $\beta_0'$ and $\beta_1'$, where $\beta_0'$ has length at most $k$ and $\beta_1'$ is an immersed path into $\alpha\vert G_t$. By choosing the maximal length legal subsegments containing the two endpoints of these segments, we may alternately subdivide $\beta\vert G_t$ into $4$ subsegments $\beta_0p\beta_1q$ separated at illegal turns, where $\beta_1$ is an immersed path into $\alpha\vert G_t$, $\beta_0$ has length at most $k$, and $p$ and $q$ are both legal. (Up to $3$ of these segments may be degenerate, as happens in the case that $\beta\vert G_t$ is itself legal).

Since the endpoints of $\beta_1$ are at illegal turns, we may unambiguously talk about the legal length $\leg(\beta_1)$ of $\beta_1$ (defined in the same way as for conjugacy classes). Since $\beta_1$ is an immersed subpath in $\alpha\vert G_t$ and $t\ge \mathrm{right}_\gamma(\alpha)$ we see that every subpath of $\beta_1$ of length at least $\illegality$ contains a legal segment of length at least $3$. Therefore, if we subdivide $\beta_1$ into $n = \floor{\len(\beta_1)/\illegality}$ subsegments of equal length $\len(\beta_1)/n \ge \illegality$, we see that each subsegment contains a legal segment of length $3$. Therefore
\[\leg(\beta_1) \ge 3n \ge 3\left(\tfrac{\len(\beta_1)}{\illegality}-1\right) = \tfrac{3}{\illegality}\len(\beta_1)-3.\]
Hence it follows that 
\begin{align*}
\leg(\beta\vert G_t) &\ge \leg(\beta_1) + \len(p)+\len(q) \ge \tfrac{3}{\illegality}\len(\beta_1) + \len(p)+\len(q)-3\\
&\ge \tfrac{3}{\illegality}(\len(\beta_1) + \len(p)+\len(q)) - 3 \ge \tfrac{3}{\illegality}(\len(\beta\vert G_t) -k) - 3.
\end{align*}
Since $\len(\beta\vert G_t) \ge 3k + 3\illegality$ by hypothesis, we may conclude
\[\leg(\beta\vert G_t) \ge \tfrac{3}{\illegality}\len(\beta\vert G_t) - (\tfrac{3k}{\illegality}+3) \ge \tfrac{3}{\illegality}\len(\beta\vert G_t) - \tfrac{1}{\illegality}\len(\beta\vert G_t) = \tfrac{2}{\illegality}\len(\beta\vert G_t).\qedhere\]
\end{proof}

\begin{corollary}
\label{cor:contained_guys_flare}
Let $\gamma\colon\I\to \os$ be a folding path, and suppose that $\beta$ is $k$--almost contained in $\alpha$ at $G_s = \gamma(s)$. If $\len(\beta)\ge 3k+3\illegality$ and $s\ge \mathrm{right}_\gamma(\alpha)$, then for all $t\in \I$ with $t\ge s$ we have
\[\len(\beta\vert G_t) \ge \frac{2}{3\illegality}\len(\beta\vert G_s)e^{t-s}.\]
\end{corollary}
\begin{proof}
By \Cref{lem:legal_flare} and \Cref{lem:legal_containment_flare} we immediately see that
\begin{align*}
\len(\beta\vert G_t) &\ge \leg(\beta\vert G_t) \ge \frac{1}{3} \leg(\beta\vert G_s)e^{t-s} \ge \frac{2}{3\illegality}\len(\beta\vert G_s) e^{t-s}.\qedhere
\end{align*}
\end{proof}

\subsection{Flaring away from length minimizers}
\label{sec:length_minimizers}

Given a subgroup $\Gamma\le \Out(\free)$, a point $R\in \os$, and a conjugacy class $\alpha$ of $\free$, we write
\[\minlen{\Gamma}{R}(\alpha)\colonequals \inf\{\len(\alpha\vert g\cdot R) : g\in \Gamma\}\]
for the infimal length of the conjugacy class $\alpha$ on the orbit $\Gamma \cdot R$. Observe that $\minlen{\Gamma}{R}(\alpha)$ is positive since it is bounded below by the injectivity radius of $R$. While this infimal length in principle need not be attained at any orbit point, we may nevertheless be assured that the set
\[\minset{\Gamma}{R}(\alpha) \colonequals \{g\in \Gamma : \len(\alpha\vert g\cdot R) \le 2\minlen{\Gamma}{R}(\alpha)\}\]
is nonempty.

We now come to the main technical lemma of \S\ref{sec:orbit_flaring}, showing that if $\Gamma$ is convex cocompact and $\beta$ is long and almost contained in $\alpha$ at a point of $\minset{\Gamma}{R}(\alpha)$, then the length of $\beta$ in the orbit $\Gamma\cdot R$ grows uniformly exponentially with the distance from $\minset{\Gamma}{R}(\alpha)$. In fact we show something slightly stronger than this:

\begin{lemma}
\label{lem:flare_away_from_minimizer2}
Suppose that $\Gamma\le\Out(\free)$ is finitely generated with word metric $d_\Gamma$ and that $\Gamma$ qi-embeds into $\fc$. For every $R\in \os$ there exist constants  $\lambda > 1$ and $C>0$ such that for every $k > 0$ there is some $L_0 > 0$ with the following property: 
Let $\alpha$ be a simple conjugacy class of $\free$ and let $g_0\in \Gamma$ lie in $\minset{\Gamma}{R}(\alpha)$. Suppose further that $g \in \Gamma$ lies on a geodesic from $g_0$ to $h \in \Gamma$, i.e. $d_\Gamma (g_0,h) = d_\Gamma(g_0,g) + d_\Gamma(g,h)$.
If $\beta\in \free$ is $k$--almost contained in $\alpha$ at $g\cdot R$ and $\len(\beta \vert g\cdot R)\ge L_0$, then  
\[\len(\beta \vert h\cdot R) \ge C\lambda^{d_\Gamma(g,h)}\len(\beta\vert g\cdot R) .\]
\end{lemma}

\begin{proof}
Choose a free factor $A$ in $\fproj(R)$. By assumption, the assignment $\varphi\mapsto \varphi\cdot A$ gives a quasi-isometric embedding $\Gamma\to \fc$. Since $\fproj$ is coarsely Lipschitz (\Cref{lem:fproj_lipschitz}), it follows from $\Gamma$--equivariance that there is some $K$ so that the assignment $\varphi\mapsto \varphi\cdot R$ gives a $K$--quasi-isometric embedding $\Gamma\to \os$. Recall this means that
\begin{equation}
\label{eqn:Gamma_qi_into_os}
\tfrac{1}{K}d_\Gamma(\varphi,\psi) - K \le d_\os(\varphi\cdot R,\psi\cdot R) \le K d_\Gamma(\varphi,\psi)+K
\end{equation}
for all $\varphi,\psi\in \Gamma$. Since $d_\Gamma(g_0,h) = d_\Gamma(g_0,g)+d_\Gamma(g,h)$, we may choose a geodesic $(\varphi_0,\dotsc,\varphi_N)$ in $(\Gamma,d_\Gamma)$ from $g_0 = \varphi_0$ to $h = \varphi_N$ with $g = \varphi_j$ for some $0 \le j \le N$. It then follows from \eqref{eqn:Gamma_qi_into_os} that the map $\gamma_0\colon[0,N]\to \os$ defined by $\gamma_0(s) = \varphi_i\cdot R$ for $s\in[i,i+1)\cap [0,N]$ is a (directed) $3K$--quasigeodesic. As described in \cite[Proposition 2.11]{BFhyp} and \cite[Theorem 5.6]{FMout}, we may build a geodesic $\hat\gamma\colon\hat\I\to \os$ from $g_0\cdot R$ to $h\cdot R$ which is a concatenation of a \emph{rescaling path} followed by a folding path. By \Cref{th:DT_2}, it follows that there exists constants $A,\epsilon > 0$ and $K'\ge 1$ depending only on $K$ and the injectivity radius of $R$ such that $\gamma_0([0,N])$ and $\hat\gamma(\hat\I)$ have Hausdorff distance at most $A$, that $\hat\gamma(\hat\I)\subset \os_\epsilon$, and that $\fproj\circ\hat\gamma\colon \hat\I\to \fc$ is a $K'$--quasigeodesic. Furthermore, the rescaling portion of $\hat\gamma$ has length at most $\log(2/\epsilon)$ by \cite[Lemma 2.6]{DT1}. If $\gamma\colon \I\to \os$ denotes the folding path portion of $\hat\gamma$, then after replacing $A$ by $A+\sym_\epsilon \log(2/\epsilon)$ (where $\sym_\epsilon$ is as in \Cref{lem:asymmetry}) we may conclude that $\dhaus\big(\gamma(\I),\{\varphi_0\cdot R,\dotsc, \varphi_N\cdot R\}\big) \le A$, that  $\dsym(\gamma(\Imin),g_0\cdot R) \le A$, that $\gamma(\Ipl) = h\cdot R$,  and that $\fproj\circ\gamma\colon \I\to \fc$ is a $K'$--quasigeodesic. In particular, since $g = \varphi_j$, there is some $\I_0\in \I$ for which $\dsym(\gamma(\I_0),g\cdot R) \le A$.

Conversely, for each $t\in \I$ there is some $i$ so that $\dsym(\gamma(t),\varphi_i\cdot R)\le A$ and consequently $\len(\alpha\vert \varphi_i\cdot R) \le e^A \len(\alpha\vert \gamma(t))$. Similarly $\len(\alpha\vert \gamma(\Imin)) \le e^A \len(\alpha \vert g_0 \cdot R)$.
Furthermore, since $g_0\in\minset{\Gamma}{R}(\alpha)$, the definition of $\minset{\Gamma}{R}(\alpha)$ gives $\len(\alpha\vert g_0\cdot R) \le 2\len(\alpha\vert \varphi_i\cdot R)$.  Hence, for every $t\in \I$ we find that
\begin{equation}
\label{eqn:left_endpoint_pretty_large} 
\len(\alpha\vert \gamma(\Imin)) \le 2e^{2A} \len(\alpha\vert \gamma(t)).
\end{equation}

Let us analyze the location of $\mathrm{right}_\gamma(\alpha)$ in $\I$. Set $D' = K'(\fcbound + K')$, where $\fcbound$ is the constant defined after \Cref{prop:BF_strong_contract}. Then all $s,t\in \I$ with $\abs{s-t} \ge D'$ satisfy $d_\fc(\gamma(s),\gamma(t))\ge \fcbound$. We claim that
\[\mathrm{left}_\gamma(\alpha) \le \Imin + D'\left(\tfrac{2A}{\log 2}+2\right).\]
Indeed, if this is false, then $\alpha\vert \gamma(t)$ is illegal for all $t\in [\Imin,\Imin + D'm]$, where $m = \ceil{\frac{2A}{\log 2} + 1} \le \frac{2A}{\log 2} + 2$. But then \Cref{lem:illegal_flare} would imply that $\len(\alpha\vert \gamma(\Imin)) > 2^m \len(\alpha\vert\gamma(\Imin+D'm))$, contradicting \eqref{eqn:left_endpoint_pretty_large} since $2^m \ge 2e^{2A}$. We also know that $\mathrm{right}_\gamma(\alpha) \le \mathrm{left}_\gamma(\alpha) + D'$ by \Cref{prop:BF_strong_contract} and our choice of $D'$. Therefore we conclude
\begin{equation}
\label{eqn:immediacy_of_right}
\mathrm{right}_\gamma(\alpha) \le \Imin + D'\left(\tfrac{2A}{\log 2} +3\right) \le \I_0 + D'\left(\tfrac{2A}{\log 2} +3\right).
\end{equation}

Set $r_0\colonequals\max\{\I_0,\mathrm{right}_\gamma(\alpha)\}$, so that $\mathrm{right}_\gamma(\alpha)\le r_0\le \I_+$. 
Thus $\dsym(\gamma(\I_0),\gamma(r_0))\le \sym_\epsilon D'(\frac{2A}{\log 2}+3)$.
Combining this with $\dsym(\gamma(\I_0),g\cdot R)\le A$ then gives $\dsym(\gamma(r_0),g\cdot R)\le E$, where $E\colonequals \sym_\epsilon D'(\tfrac{2A}{\log 2} +3) + A$. 

Suppose now that $\beta$ is $k$--almost contained in $\alpha$ at $g\cdot R$ for the given constant $k$. By \Cref{lem:nearby containment}, we know that $\beta$ is $k'$--almost contained $\alpha$ at $\gamma(r_0)$ for some constant $k'$ depending only on $k$ and $E$. Define $L_0 \colonequals (3k'+3\illegality)e^{E}$ so that the additional assumption $\len(\beta\vert g\cdot R)\ge L_0$ will moreover imply
\[\len(\beta\vert \gamma(r_0)) \ge \len(\beta\vert g\cdot R) e^{-E} \ge 3k' + 3\illegality.\]
Thus if $\beta$ is $k$--almost contained in $\alpha$ at $g\cdot R$ and $\len(\beta\vert g\cdot R) \ge L_0$, we may apply \Cref{cor:contained_guys_flare} to conclude
\begin{equation}\label{eqn:beta_flares}
\len(\beta\vert \gamma(t)) \ge \frac{2}{3\illegality}\len(\beta \vert \gamma(r_0)) e^{(t-r_0)} \ge \frac{2}{3\illegality e^{E}}\len(\beta\vert g\cdot R) e^{(t-r_0)}
\end{equation}
for all $t \ge r_0$ (this is valid because $r_0\ge \mathrm{right}_\gamma(\alpha)$).

We now use \eqref{eqn:beta_flares} to prove the proposition. Since $\Ipl \ge r_0$ by construction, \eqref{eqn:beta_flares} immediately gives
\[\len(\beta\vert h\cdot R) = \len(\beta\vert \gamma(\Ipl)) \ge \frac{2}{3\illegality e^{E}} \len(\beta\vert g\cdot R) e^{(\Ipl - r_0)}.\]
Since $\dsym(g\cdot R, \gamma(r_0))\le E$ and $\gamma$ is a directed geodesic, we have that $d_\os(g\cdot R,h\cdot R) \le E + (\Ipl-r_0)$. We also know $d_\os(g\cdot R,h\cdot R) \ge \tfrac{1}{K}d_\Gamma(g,h)-K$ by Equation \eqref{eqn:Gamma_qi_into_os}. Therefore $(\Ipl-r_0) \ge \tfrac{1}{K}d_\Gamma(g,h) - K -E$ and so 
\[\len(\beta\vert h\cdot R) \ge \frac{2}{3\illegality e^{2E+K}} \len(\beta\vert g\cdot R) \left(e^{1/K}\right)^{d_\Gamma(g,h)}\]
as desired. 
Thus the conclusion of the proposition holds with $\lambda = e^{1/K}$ and $C = \tfrac{2}{3\illegality e^{2E+K}}$.
\end{proof}

\section{Distortion within fibers of $E_\Gamma \to \Gamma$}
\label{sec:undistortion_in_fibers}

Fix a finitely generated subgroup $\Gamma \le \Out(\free)$ for which the orbit map $\Gamma \to \CS$ is a quasi-isometric embedding. Then by \Cref{th:DT1} and \Cref{th:qi_into}, the corresponding extension $E_\Gamma$ is hyperbolic. In this section, we establish \Cref{th:scott_swarup} which shows that if $H \le \free$ is finitely generated and of infinite index, then $H$ is quasiconvex (and hence undistorted) as a subgroup of $E_\Gamma$.  This will follow from the structural result \Cref{width}, which will be used in \S\ref{sec:width} to characterize which hyperbolic extensions of $\free$ induce convex cocompact subgroups of $\Out(\free)$.

\subsection{The Cayley graph bundle}\label{sec:metric_bundle}

To this end, we first recall some notation and results from \cite[\S\S7--8]{DT1} describing the structure of $E_\Gamma$. Fix a finite generating set $S = \{ s_1, \ldots, s_n \}$ of $\Gamma$ and a free basis $X = \{x_1, \ldots, x_r\}$ for $\free$. Recalling that $E_\Gamma$ is the preimage of $\Gamma$ under the quotient map $\Aut(\free)\to\Out(\free)$, we choose a lift $t_i \in \Aut(\free)$ of $s_i$ for each $1\le i \le n$. In general, we will use the notation $\tilde{g}\in \Aut(\free)$ to denote a lift of $g\in \Gamma$ to $E_\Gamma$. We also write $i_{x}\in \Aut(\free)$ for the inner automorphism given by conjugation by $x\in \free$, i.e.,  $i_x(a) = x a x^{-1}$ for $a \in \free$. Note that $\varphi i_x \varphi^{-1} = i_{\varphi(x)}$ for each $x\in\free$ and $\varphi\in \Aut(\free)$.

As a subgroup of $\Aut(\free)$, $E_{\Gamma}$ is thus generated by the set $W = \{i_{x_1}, \ldots i_{x_r},t_1, \ldots, t_n\}$. That is
\[E_{\Gamma} = \langle i_{x_1}, \ldots i_{x_r},t_1, \ldots, t_n \rangle  \le \Aut(\free).\]
For convenience, set $\hat{X} = \{ i_{x_1}, \ldots, i_{x_r} \}$ and $\hat\free = \langle \hat{X} \rangle$, so that $\hat \free$ is the isomorphic image of $\free$ in $\Aut(\free)$. Note that $\hat\free$ is also the kernel of the homomorphism $E_{\Gamma} \to \Gamma$. We also set $\tilde S = \{t_1,\dotsc,t_n\}$.

Let $T = \cay{X}{\free}$,  $\bund = \cay{W}{E_{\Gamma}}$, and $\base = \cay{S}{\Gamma}$, where $\cay{\cdot}{\cdot}$ denotes the Cayley graph with the specified generating set equipped with the path metric in which each edge has length one. We respectively view $\hat\free\cong\free$, $E_\Gamma$, and $\Gamma$ as the $0$--skeletons of the simplicial complexes $T$, $\bund$, and $\base$.
Set $\rose$ to be the standard rose on the generating set $X$ so that $\rose = T / \free$. There is then an obvious equivariant simplicial map
\[p\colon \bund \to \base \]
extending the surjective homomorphism $E_{\Gamma} \to \Gamma$;  note that $p$ sends edges of $\bund$ to either vertices or edges of $\base$ depending on whether the edge corresponds to a generator in $\hat X$ or $\tilde S$, respectively. For each $b\in \Gamma$, we see that the preimage $T_{b} = p^{-1}(b)$ is the simplicial tree (isomorphic to $T$) with vertices labeled by the coset $\tilde{b}\hat\free$ ($\tilde{b}$ any lift of $b$) and edges labeled by $\hat X$. We write $d_b$ for the induced path metric on the fiber $T_b$ over $b\in \Gamma$. 
By a \define{$k$--qi section} of $p\colon \bund\to\base$, we simply mean a $k$--quasi-isometric embedding $\sigma\colon \base\to\bund$ such that $(p\circ \sigma)(g) = g$ for every $g\in \Gamma$ (i.e., for every vertex of $\base$). By Mosher's ``quasi-isometric section lemma'' \cite{MosherExtensions}, there exists a constant $\qisec\ge 1$ (depending only on the bundle $p\colon \bund\to\base$) such that for every $b\in \Gamma$ and vertex $v\in T_b$ one may build a $\qisec$--qi section $\sigma\colon \base\to\bund$ with $\sigma(b) = v$.

As discussed in \cite[\S7]{DT1} (see also \cite[Example 1.8]{MjSardar}), $p\colon \bund\to\base$ is an example of the \emph{ metric graph bundle} construction developed by Mj and Sardar in \cite{MjSardar}. In particular, there is a \define{metric properness function} $f\colon \N\to \N$ such that $d_b(u,v)\le f(d_{\bund}(u,v))$ for all $b\in \Gamma$ and all vertices $u,v\in T_b$ \cite[Lemma 7.2]{DT1}. 
We moreover observe that if group elements $u,v\in E_\Gamma$ lie in the same fiber $T_b$, then $u\inv v\in \hat\free$ and the fiberwise distance satisfies $d_b(u,v) = \abs{u\inv v}_{\hat X}$ \cite[Lemma 7.1]{DT1}. Writing $u\inv v = i_a$ for the appropriate $a\in \free$ and defining $\bblip\ge 1$ to be the maximal bilipschitz constant of the automorphisms in $\tilde{S}$, we deduce the inequality
\begin{equation}
\label{eqn:bilipschitz_neighborhs}
d_{bs_i}(ut_i,vt_i) = \abs{t_i\inv u\inv v t_i}_{\hat X} = \abs{t_i\inv i_a t_i}_{\hat X} = \abs{i_{t_i\inv(a)}}_{\hat X} \le \bblip \abs{i_a}_{\hat X} = \bblip\, d_b(u,v)
\end{equation}
for every generator $s_i\in S$ with corresponding lift $t_i\in \tilde{S}$.

\subsection{The Width Theorem}
\label{sec:width_theorem}

Suppose that $a$ is an \emph{element} (rather than a conjugacy class) of $\free$. 
Then left multiplication by the inner automorphism $i_a\in \hat{\free}$ gives an isometry of $\bund$ that preserves each fiber $T_b$ of $p\colon \bund\to\base$. In particular, $i_a$ acts as a hyperbolic isometry of $(T_b,d_b)$ translating along a unique invariant axis. We write $\axis_b(a)$ for the axis of $i_a$ in the fiber $T_b$ and then define the \define{axis bundle of $a$} to be the union $\axis(a)\colonequals \cup_{b\in \Gamma}\axis_b(a)$.

Note that while $\axis_g(a)$ is a geodesic in the path metric space $(T_g,d_g)$, it will generally be far from being a geodesic in the whole space $\bund$. However, our next result shows that when $a$ is simple and $g\in \Gamma$ lies in the minimizing set $\minset{\Gamma}{R}([a])$ for the conjugacy class of $a$ (see \S\ref{sec:length_minimizers}), then $\axis_g(a)$ is a uniform quasigeodesic in $\bund$.

\begin{proposition}\label{lem:near_fiber}
Suppose that $\Gamma\le\Out(\free)$ qi-embeds into $\CS$ and let $p\colon \bund\to\base$ be as in \S\ref{sec:metric_bundle}. Then for every $R\in \os$ there exists a constant $Q\ge 1$ such that for any simple element $a\in \free$ and any $g\in \minset{\Gamma}{R}([a])$, the axis $\axis_g(a)$ (viewed as a map $\R\to T_g\subset \bund$) is a $Q$--quasigeodesic in $\bund$.
\end{proposition}

The proof of \Cref{lem:near_fiber} is fairly technical and will be deferred to the next section. Meanwhile, we use it to uniformly bound the ``width'' of all simple conjugacy classes of $\free$. Suppose now that $\Gamma\le\Out(\free)$ qi-embeds into $\CS$ so that the corresponding bundle $\bund$ is a hyperbolic metric space. Every element $a\in \free$ then acts (via left multiplication by $i_a\in \hat\free)$ as a hyperbolic isometry of $\bund$, and we let $a^*$ denote a biinfinite geodesic of $\bund$ joining the two fixed points $a^\infty,a^{-\infty}$ of $a$ in $\partial \bund$. Define the \define{width} of $a \in \free$ (or its conjugacy class $\alpha$) by
\[
\width(\alpha) = \width(a) \colonequals \diam_\base p(a^*).
\]

\begin{theorem}[Width Theorem]\label{width}
Suppose that $\Gamma\le\Out(\free)$ qi-embeds into $\cs$ and consider the hyperbolic extension $p\colon E_\Gamma\to \Gamma$ of $\free$. 
Then the simple conjugacy classes of $\free$ have uniformly bounded width.  That is, 
\[
\sup _\alpha \diam_{\Gamma} p(\alpha^*) < \infty
\] 
where the supremum is over simple conjugacy classes of $\free$.
\end{theorem}

\begin{proof}
Let $a$ be a simple element of $\free$ and $\alpha$ is conjugacy class. Suppose that the length of $\alpha$ is minimized over the fiber $T_g$ for $g \in  \Gamma$. Then by \Cref{lem:near_fiber}, the axis $\mathcal{A}_{g}(a)$ of $a$ in $T_g$ is a $Q$--quasigeodesic for $Q\ge 0$ not depending on $a$. As $E_\Gamma$ is hyperbolic, \Cref{prop:stability_of_quasis} provides a constant $R = R(Q)\ge 0$ so that the axis of $a$ in $E_\Gamma$ has Hausdorff distance at most $R$ from any $Q$--quasigeodesic joining its endpoints in $\partial E_\Gamma$. Hence, $d_{\mathrm{Haus}}(a^*,\mathcal{A}_{g}(a)) \le R$ and so the diameter of the image of $a^*$ in $\Gamma$ is at most $R$. Since this is independent of the conjugacy class $\alpha$, the theorem follows.
\end{proof}

In \S\ref{sec:width}, we will show that any hyperbolic extension of $\free$ in which simple elements have uniformly bounded width is an extension by a convex cocompact subgroup $\Gamma \le \Out(\free)$.

\subsection{Axis bundles and the proof of \Cref{lem:near_fiber}} \label{sec:near_fiber}
We now embark on the proof of \Cref{lem:near_fiber}. Our approach is modeled on that of Kent--Leininger in \cite{KentLein-geometric}, where they prove an analogous result in order to establish their width theorem for hyperbolic extensions of surface groups. The main idea is to use the axis bundle $\axis(a)$ to construct a Lipschitz retract from $\bund$ to $\axis_g(a)$. 

The first step of the construction utilizes the techniques that Mitra \cite{MitraCTmaps-general} has developed and used extensively to study hyperbolic group extensions.
Let $p\colon \bund\to\base$ be as in \S\ref{sec:metric_bundle} and let $a\in \free$ be a nontrivial element with corresponding axis bundle $\axis(a)\subset\bund$. Define $\fiberwise_a\colon E_\Gamma\to \axis(a)$ to be the fiber-wise closest point projection to $\axis(a)$, that is, for each vertex $x$ in the fiber $T_g$ we define $\fiberproj_a(x)$ to be the unique point in $\axis_g(a)$ minimizing the distance $d_g(x,\fiberproj_a(x))$. 

\begin{lemma}[Mitra {\cite[Lemma 3.2]{MitraCTmaps-general}}]
\label{lem:fiberwise_lipschitz}
There is a constant $\fiberlip\ge 1$ depending only on the bundle $\bund\to\base$ such that $\fiberwise_a\colon E_\Gamma\to \axis(a)$ is $\fiberlip$--Lipschitz for each element $a\in \free$. That is, for for all $u,v\in E_\Gamma$ we have
\[d_{\bund}(\fiberwise_a(u),\fiberwise_a(v)) \le \fiberlip d_\bund(u,v).\]
\end{lemma}
\begin{proof}
Mitra's proof of \Cref{lem:fiberwise_lipschitz} follows from basic hyperbolic geometry; for completeness we give a brief sketch here: By the triangle inequality it suffices to suppose $d_\bund(u,v) = 1$. Then if $u,v$ lie in the same fiber $T_g$, we immediately have $d_\bund(\fiberwise_a(u),\fiberwise_a(v))\le d_g(\fiberwise_a(u),\fiberwise_a(v)) \le 1$ by the nature of closest-point-projection in the tree $T_g$. Otherwise $u$ and $v$ lie in neighboring fibers so that $v = ut_i$ for some $t_i\in \tilde{S}$. But then one may use the uniform bilipschitz equivalence of neighboring fibers (Equation \eqref{eqn:bilipschitz_neighborhs}) to prove that $\fiberwise_a(u)t_i$ is uniformly close to $\fiberwise_a(v)$ (which is the content of \cite[Lemma 3.6]{MitraCTmaps-general}).
\end{proof}

\Cref{lem:fiberwise_lipschitz} allows us to extend $\fiberwise_a$ to a map $\fiberwise_a\colon \bund\to \axis(a)$ that is coarsely $\fiberlip$--Lipschitz.
For $a\in \free$ nontrivial, let us use the terminology \define{$k$--qi section through $\axis(a)$} to mean a $k$--quasi-isometric embedding $\sigma\colon (\base,d_\base)\to(\bund,d_\bund)$ such that $\sigma(g)\in \axis_g(a)$ for all $g\in \Gamma$. 

\begin{corollary}
\label{cor:sections_through_axis}
For any nontrivial $a\in \free$ and any vertex $v\in \axis_g(a)$, there exists a $\fiberlip\qisec$--qi section $\sigma$ through $\axis(a)$ with $\sigma(g) = v$.
\end{corollary}
\begin{proof}
Let $\sigma_0\colon \base\to\bund$ be the $\qisec$--qi section with $\sigma_0(g) = v$ provided by Mosher \cite{MosherExtensions}. Composing $\sigma_0$ with $\fiberwise_a\colon \bund\to \axis(a)$ then gives the desired $\fiberlip\qisec$--quasi-isometric embedding $\fiberwise_a\circ\sigma_0\colon \base\to\bund$.
\end{proof}

We now make a basic observation about ``well-separated'' $k$--qi sections through axis bundles.

\begin{lemma}
\label{lem:separated_sections}
There exists a constant $\separation > 0$ depending only on $\bund\to\base$ with the following property. Suppose $a\in \free$ is nontrivial, that $\sigma_1,\sigma_2$ are $\fiberlip\qisec$--qi sections through $\axis(a)$, and that $g,h\in \Gamma$ satisfy $d_\base(g,h) \le 1$. If $u\in \axis_g(a)$ lies between $\sigma_1(g)$ and $\sigma_2(g)$ on $\axis_g(a)$ with $d_g(u,\sigma_i(g))\ge \separation$ for $i=1,2$ and $v\in \axis_h(a)$ is a vertex with $d_\bund(u,v)\le \fiberlip\qisec$, then $v$ also lies between $\sigma_1(h)$ and $\sigma_2(h)$ on $\axis_h(a)$.
\end{lemma}
\begin{proof}
Define $\separation\colonequals 4E^2$, where $E = \bblip+2R_0(\bblip,0) + f(2\fiberlip\qisec + 1+R_0(\bblip,0))$, and suppose that $g,h,u,v,\sigma_1,\sigma_2$ are as in the statement of the lemma. Here $R_0$ is as in \Cref{prop:stability_of_quasis}. If $d_\base(g,h) = 0$ then $v$ is a point on the geodesic $\axis_g(a)$ within $d_g$--distance $f(\fiberlip\qisec)$ of $u$. Thus the result is immediate  since $\separation > f(\fiberlip\qisec)$. Otherwise $d_\base(g,h) = 1$ so that $h = gs_i$ for some generator $s_i\in S$. Let $t_i$ be the chosen lift in $\tilde{S}$. Define a map $\Psi$ from $\tilde{g}\hat \free$ (the vertex set of $T_g$) to $\tilde{h}\hat\free$ (the vertex set of $T_h$) by declaring $ \Psi(x)$ to be the $d_h$--closest-point-projection of $xt_i\in T_h$ to $\axis_h(a)$. Since the assignment $x\mapsto xt_i$ is $\bblip$--bilipschitz by Equation (\ref{eqn:bilipschitz_neighborhs}) and $T_h$ is $0$--hyperbolic, it follows from \Cref{prop:stability_of_quasis} and the definition of $E > \bblip+2R_0(\bblip,0)$ that $\Psi$ restricts to an $E$--quasi-isometric embedding from (the vertices of) $\axis_g(a)$ to $\axis_h(a)$. Observe also that $d_\bund(\Psi(x),x)\le 1 + R_0(\bblip,0)$ for each vertex $x\in \axis_g(a)$. 

The hypotheses on $\sigma_1(g),u,\sigma_2(g)$ now imply that $\Psi(u)$ appears between $\Psi(\sigma_1(g))$ and $\Psi(\sigma_2(g))$ on $\axis_h(a)$  with $d_h\big(\Psi(u),\Psi(\sigma_i(g)\big) \ge 3E$ for $i=1,2$. Using the triangle inequality, the hypotheses on $\sigma_1,\sigma_2,v$ with the above observation about $\Psi\vert_{\axis_g(a)}$ together give
\[d_\bund\big(\Psi(\sigma_1(g)),\sigma_1(h)\big),\,d_\bund\big(\Psi(\sigma_2(g)),\sigma_2(h)\big),\, d_\bund\big(\Psi(u),v\big)\le 2\fiberlip\qisec + 1 + R_0(\bblip,0).\]
In particular, the first inequality follows from the estimate
\begin{align*}
d_\bund\big(\Psi(\sigma_1(g)),\sigma_1(h)\big) &\le d_\bund\big(\Psi(\sigma_1(g)),\sigma_1(g)\big) + d_\bund\big(\sigma_1(g),\sigma_1(h)\big) \\
&\le 1 + R_0(\bblip,0) + \fiberlip\qisec \cdot d_\base(g,h) +\fiberlip\qisec\\
&\le 2\fiberlip\qisec + 1 + R_0(\bblip,0),
\end{align*}
and the other two are similar.
By metric properness, we may thus conclude that
\[d_h\big(\Psi(\sigma_1(g)),\sigma_1(h)\big),\,d_h\big(\Psi(\sigma_2(g)),\sigma_2(h)\big),\, d_h\big(\Psi(u),v\big)\le f(2\fiberlip\qisec + 1 + R_0(\bblip,0)) \le E.\]
Therefore the triangle inequality shows that $v$ lies between $\sigma_1(h)$ and $\sigma_2(h)$ on $\axis_h(a)$, as desired.
\end{proof}

Next, when $a\in \free$ is simple the flaring property established in \S\ref{sec:length_minimizers} translates into the following estimate for well-separated qi-sections through $\axis(a)$. Let $\rose\in \os$ denotes the marked graph $\cay{X}{\free}/\free$ equipped with the metric in which each edge has length $\nicefrac{1}{\rank(\free)}$

\begin{proposition}
\label{lem:qi_sections_flare}
Suppose that $\Gamma\le\Out(\free)$ qi-embeds into $\fc$ and let $p\colon \bund\to\base$ be as in \S \ref{sec:metric_bundle}. For every $\bar K\ge 1$ there exist $D,D_0  > 0$ and $\eta > 1$ so that the following holds. Suppose that $a\in \free$ is simple and that $\sigma,\sigma'$ are $\bar K$--qi-sections through $\axis(a)$ with $d_{g}(\sigma(g),\sigma'(g))\ge D_0$ for some element $g\in \minset{\Gamma}{\rose}([a])$. Then for all $h\in \Gamma$ we have:
\[
d_h(\sigma(h), \sigma'(h)) \ge D \eta^{d_\Gamma(g,h)}\, d_{g}(\sigma(g),\sigma'(g)).
\]
\end{proposition}
\begin{proof}
Le $C$, $\lambda$, and $L_0$ be the constants obtained by applying \Cref{lem:flare_away_from_minimizer2} to the orbit $\Gamma\cdot \rose$ with $k = 1$, and let $f$ be the metric properness function for the graph bundle $\bund\to\base$. Fix $N\ge 1$ large enough so that $C\lambda^N > 4$ and define
\[D_0 \colonequals \max\left\{2,\; 2\rank(\free)L_0,\; 8C\inv f(6\bar KN)\right\}.\]

Choose $h\in \Gamma$ arbitrarily, and let $g = g_0,\dotsc,g_m=h$ be a geodesic from $g$ to $h$ in $\Gamma$. Let us write $\sigma_i \colonequals \sigma(g_i)\in T_{g_i}$ and $\sigma'_i\colonequals \sigma'(g_i)\in T_{g_i}$ for the value of the two $\bar K$--qi-sections in the fiber $T_{g_i}$ of $p\colon\bund\to\base$ over $g_i$.
Choose any vertices $g_i,g_j$ along our geodesic with $i\le j \le i+2N$, and suppose temporarily that $d_{g_i}(\sigma_i,\sigma_i')\ge D_0$. Recall from \S\ref{sec:metric_bundle} that $T_{g_i} = p^{-1}(g_i)$ is a simplicial tree whose edges are labeled by the free basis $\hat X$ of $\hat \free$. With respect to this basis, the element $i_c = \sigma_i\inv\sigma'_i \in \hat \free$ may not by cyclically reduced. However, there is some $x\in \hat X$ so that $i_b = \sigma_i\inv\sigma'_ix \in \hat \free$ is cyclically reduced. Set $z_i = \sigma_i$ and $z_i' = \sigma_i'x$, so that the geodesic edge path $[z_i,z_i']$ in $T_{g_i}$ is labeled by the cyclically reduced word $i_b$ of $\hat\free$ with the properties that $\norm{i_b}_{\hat X} = \abs{i_b}_{\hat X}$ and that $\abs{i_b}_{\hat X}$ differs from $d_{g_i}(\sigma_i,\sigma'_i) = \abs{\sigma_i\inv,\sigma'_i}_{\hat X}$ by at most $1$.

Choosing a lift $\tilde{g}_i\in \Aut(\free)$ of $g_i\in \Out(\free)$, the action $\tilde{g}_i$ on $\bund$ restricts to a simplicial automorphism from $T_1$ to $T_{g_i}$ that respects the edge labeling and thus gives an identification of $T_1 = \cay{\hat X}{\hat\free}$ with $T_{g_i}$. With respect to this identification, the element $\tilde{g}_i(b)\in \free$ acts on $T_{g_i}$ the same way that $i_b$ acts on $\cay{\hat X}{\hat \free}$. Therefore, since the labeled edge path $[z_i,z_i']$ is a fundamental domain of the axis of $i_b$ in $\cay{\hat X}{\hat \free}$, it follows that $[z_i,z_i']\subset T_{g_i}$ is a fundamental domain of the axis for $\tilde{g}_i(b)$ acting on $T_{g_i}$. 
Letting $\beta$ and $\alpha$ denote the conjugacy classes of $b$ and $a$, respectively, it follows that $\tilde{g}_i(\beta) = g_i(\beta)$ is $1$--almost contained in $\alpha$ at $g_i\cdot\rose$ (since $\sigma_i,\sigma'_i\in \axis_{g_i}(a)$ by construction and all edges in the universal cover of $g_i\cdot \rose$ have length $\nicefrac{1}{\rank(\free)}\le 1$). Since
\[\len(g_i(\beta)\vert g_i\cdot \rose) = \len(\beta\vert \rose) = \tfrac{1}{\rank(\free)}\norm{i_b}_{\hat X} \ge \tfrac{1}{\rank(\free)}\left(d_{g_i}(\sigma_i,\sigma'_i) - 1\right) \ge \tfrac{1}{2\rank(\free)}d_{g_i}(\sigma_i,\sigma'_i) \ge L_0\]
by the assumption $d_{g_i}(\sigma_i,\sigma'_i)\ge D_0$, we may apply \Cref{lem:flare_away_from_minimizer2} to conclude
\begin{equation*}
\label{eqn:conjugacy_lower_bound}
\len(g_i(\beta)\vert g_j\cdot \rose) \ge C\lambda^{j-i}\len(g_i(\beta)\vert g_i\cdot \rose) \ge \tfrac{C}{2\rank(\free)}\lambda^{j-i}d_{g_i}(\sigma_i,\sigma'_i).
\end{equation*}

For each $i < p \le j$, set $s_p = g_{p-1}\inv g_{p}\in S$ and let $t_p\in \tilde{S}$ be the chosen lift of $s_p$ in the generating set $W$ of $E_\Gamma$. For $i \le p \le j$ let us also define $z_p = z_i t_{i+1}\dotsb t_p$ and $z'_p = z'_i t_{i+1}\dotsb t_p$, both of which are points over $g_p = g_i s_{i+1}\dotsb s_p$. Since $g_i,g_{i+1},\dotsc, g_j$ is a geodesic in $\Gamma$, it follows that $z_i,\dotsc,z_j$ and $z'_i,\dotsc,z'_j$ are both geodesics in $E_\Gamma$ and thus that $d_{\bund}(z_i,z_j) = d_\bund(z'_i,z'_j) = j-i \le 2N$. Observe now that
\[z_j\inv z'_j = \left(z_i t_{i+1}\dotsb t_{j}\right)\inv\left(z'_i t_{i+1}\dotsb t_{j}\right) = \left(t_j\inv \dotsb t_{i+1}\inv\right)z_i\inv z'_i\left(t_{i+1}\dotsb t_{j}\right) = \varphi i_b \varphi\inv = i_{\varphi(b)}\in \hat\free,\]
where $\varphi$ is the specific lift $\varphi  =t_j\inv \dotsb t_{i+1}\inv\in \Aut(\free)$ of $g_j\inv g_i = s_j\inv \dotsb s_{i+1}\inv$. In particular, we see that the distance between $z_j$ and $z'_j$ in the fiber $T_{g_j}$ satisfies
\begin{equation}\label{eqn:dist_between_zs}
d_{g_j}(z_j,z'_j) = \abs{i_{\varphi(b)}}_{\hat X} \ge \norm{i_{\varphi(b)}}_{\hat X} = \rank(\free)\len(\varphi(\beta)\vert \rose) = \rank(\free)\len(g_i(\beta)\vert g_j\cdot \rose)  \ge \tfrac{C}{2}\lambda^{j-i} d_{g_i}(\sigma_i,\sigma'_i).
\end{equation}

Let us now compare this distance to $d_{g_j}(\sigma_j,\sigma_j')$. Since $\sigma,\sigma'$ are $\bar K$--qi-sections, the quantities $d_{\bund}(\sigma_i,\sigma_j)$ and $d_\bund(\sigma'_i,\sigma'_j)$ are bounded by $\bar K(j-i) + \bar K \le 2\bar K N+\bar K$. Since we also have $d_\bund(\sigma_i,z_i),d_\bund(\sigma'_i,z'_i)\le 1$ by construction, the triangle inequality thus gives
\begin{align*}
d_\bund(z_j,\sigma_j) &\le d_\bund(z_j,z_i) +d_\bund(z_i, \sigma_i)+ d_\bund(\sigma_i,\sigma_j)\\
&\le 2N + 1 + 2\bar KN + \bar K \le 6\bar K N,
\end{align*}
and similarly, $d_\bund(z_j',\sigma_j') \le 6\bar K N$.
By metric properness of the graph bundle $\bund\to\base$, it follows that 
\linebreak $d_{g_j}(z_j,\sigma_j), d_{g_j}(z'_j, \sigma'_j) \le f(6\bar KN)$. Combining with \eqref{eqn:dist_between_zs} and using $d_{g_i}(\sigma_i,\sigma'_i)\ge D_0 \ge 8C\inv f(6\bar KN)$ and $\lambda^{j-i}\ge 1$, we conclude that
\begin{align*}
d_{g_j}(\sigma_j,\sigma'_j) 
&\ge d_{g_j}(z_j,z'_j) - d_{g_j}(z_j,\sigma_j) - d_{g_j}(z'_j, \sigma'_j)  \\
&\ge \tfrac{C}{2} \lambda^{j-i} d_{g_i}(\sigma_i,\sigma'_i) - 2f(6\bar K N)\\
&\ge \left(\lambda^{j-i} - \tfrac{1}{2}\right)\tfrac{C}{2}d_{g_i}(\sigma_i,\sigma'_i)\ge \tfrac{C}{4}\lambda^{j-i} d_{g_i}(\sigma_i,\sigma'_i).
\end{align*}

To summarize, we have now shown that the implication
\begin{equation}
\label{eqn:implication}
d_{g_i}(\sigma_i,\sigma'_i) \ge D_0 \implies d_{g_j}(\sigma_j,\sigma_j') \ge \tfrac{C}{4}\lambda^{j-i}d_{g_i}(\sigma_i,\sigma'_i)
\end{equation}
holds for any pair of vertices $g_i,g_j$ on the geodesic $g=g_0,\dotsc,g_m=h$  with $i\le j\le i+2N$. Suppose now that $d_g(\sigma(g),\sigma'(g))\ge D_0$. If $d_\Gamma(g,h) >  N$, we may then break the geodesic $g_0,\dotsc,g_m$ into $\floor{d_\Gamma(g,h)/N}\ge \tfrac{1}{2N}d_\Gamma(g,h)$ pieces that each have length between $N$ and $2N$ and inductively apply the estimate \eqref{eqn:implication} to conclude
\[d_{h}(\sigma(h),\sigma'(h))\ge \left(\tfrac{C}{4}\lambda^N\right)^{\floor{d_\Gamma(g,h)/N}}d_g(\sigma(g),\sigma'(g)) \ge  \eta^{d_\Gamma(g,h)} d_{g}(\sigma(g),\sigma'(g)),\]
where $\eta \colonequals \left(\tfrac{C\lambda^N}{4}\right)^{\nicefrac{1}{2N}} > 1$. Otherwise, $d_\Gamma(g,h)\le N$ and \eqref{eqn:implication} immediately gives the desired bound
\[d_h(\sigma(h),\sigma'(h)) \ge \left(\tfrac{C}{4\eta^N}\right) \eta^{d_\Gamma(g,h)} d_g(\sigma(g),\sigma'(g)).\]
Setting $D = \tfrac{C}{4\eta^N}$ completes the proof.
\end{proof}

With these tools in hand, we are now prepared to give the

\begin{proof}[Proof of \Cref{lem:near_fiber}]
Let $\qisec$, $\fiberlip$, and $\separation$ be the constants provided by Mosher's quasi-isometric section lemma, \Cref{lem:fiberwise_lipschitz}, and \Cref{lem:separated_sections} (all which depend only on the bundle $\bund\to\base$). Let $D,D_0> 0$ be the constants obtained by applying \Cref{lem:qi_sections_flare} with $\bar K = \fiberlip\qisec$, and fix a constant $M >D_0 + (f(\fiberlip)+\separation)/ D$. 

Let $a\in \free$ and $g\in \minset{\Gamma}{\rose}([a])$ be as in the statement of the proposition. We may use \Cref{cor:sections_through_axis} to construct an infinite family $\{\Sigma_i\}_{i\in \Z}$ of $\fiberlip\qisec$--qi sections through $\axis(a)$ with the property that and $d_g(\Sigma_i(g),\Sigma_j(g)) = M\abs{i-j}$ for all $i,j\in \Z$. Notice that this forces the points $\dotsc,\Sigma_{-1}(g),\Sigma_0(g),\Sigma_1(g),\dotsc$ to be linearly ordered along the axis $\axis_g(a)$. Furthermore, for all $i\ne j$ we have $d_g(\Sigma_i(g),\Sigma_j(g))\ge D_0$ so that we may apply \Cref{lem:qi_sections_flare} to conclude $d_h(\Sigma_i(h),\Sigma_j(h))\ge DM  > \separation + f(\fiberlip)$ for all $h\in \Gamma$.

From this we claim that the sections $\{\Sigma_i\}$ are \emph{consistently ordered} in each fiber, meaning that if $\Sigma_j(h)$ appears between $\Sigma_i(h)$ and $\Sigma_k(h)$ in the axis $\axis_h(a)$ for some $h\in \Gamma$, then the same holds for every $h'\in \Gamma$. Indeed, if $\Sigma_j(h)$ appears between $\Sigma_i(h)$ and $\Sigma_k(h)$ in $\axis_h(a)$, then applying \Cref{lem:separated_sections} with $\sigma_1 = \Sigma_i$, $\sigma_2 = \Sigma_k$, $u= \Sigma_j(h)$ shows that $v = \Sigma_j(h')$ appears between $\Sigma_i(h')$ and $\Sigma_k(h)$ in any neighboring fiber $h'$; thus the consistently ordered conclusion follows by induction.

We now use the sections $\{\Sigma_i\}_{i\in \Z}$ to define a map $\mathfrak{q}_a\colon \axis(a)\to \axis_g(a)$, as follows. For each $h\in \Gamma$, the sections $\Sigma_i$ partition the geodesic $\axis_h(a)\cong \R$ into infinitely many, disjoint, half open intervals $Z_j^h\colonequals[\Sigma_{j}(h),\Sigma_{j+1}(h))$. Define the map $\mathfrak{q}_a$ by sending the interval $Z_j^h = [\Sigma_j(h),\Sigma_{j+1}(h))$ to the point $\Sigma_j(g)\in \axis_g(a)$ (so the image of $\mathfrak{q}_a$ is the set $\{\Sigma_j(g)\}_{j\in \Z}$). Next define $\Pi_a\colon E_\Gamma\to \axis_g(a)$ to be the composition $\Pi_a = \mathfrak{q}_a\circ\fiberwise_a$.

\begin{claim*}
The map $\Pi_a\colon E_\Gamma\to (\axis_g(a),d_g)$ is a coarse $3M$--Lipschitz retraction onto $\axis_g(a)$, meaning that for all $u,v\in E_\Gamma$ and each vertex $x\in \axis_g(a)$ we have
\[d_g(\Pi_a(u),\Pi_a(v))\le 3M d_\bund(u,v)\qquad\text{and}\qquad d_g(x,\Pi_a(x))\le 3M.\]
\end{claim*}
Indeed, for any vertex $x\in \axis_g(a)$ we have $\fiberwise_a(x) = x$, by definition, so that $\Pi_a(x) = \Sigma_j(g)$ where $j\in \Z$ is the unique integer such that $x\in Z_j^g=[\Sigma_j(g),\Sigma_{j+1}(g))$. Thus $d_g(x,\Pi_a(x))\le M$ since $d_g(\Sigma_j(g),\Sigma_{j+1}(g)) = M$ by construction of the family $\{\Sigma_i\}$. To complete the proof of the claim, it thus suffices to prove the bound $d_g(\Pi_a(u),\Pi_a(v))\le 3M$ for all $u,v\in E_\Gamma$ with $d_\bund(u,v)=1$. First suppose $u,v\in T_h$ for some $h\in \Gamma$. Then $\fiberwise_a(u),\fiberwise_a(v)\in \axis_h(a)$ satisfy
\[d_h(\fiberwise_a(u),\fiberwise_a(v)) \le f\big(d_\bund(\fiberwise_a(u),\fiberwise_a(v))\big) \le f(\fiberlip)\]
by \Cref{lem:fiberwise_lipschitz}. Since $d_h(\Sigma_m(h),\Sigma_n(h))\ge \separation +f(\fiberlip)$ for all $m\ne n$, it follows that if $\fiberwise_a(u)$ lies in $Z_i^h$ and $\fiberwise_a(v)$ lies in $Z_j^h$, then $\abs{i-j}\le 1$. Thus $d_g(\Pi_a(u),\Pi_a(v)) = d_g(\Sigma_i(g),\Sigma_j(g)) = M\abs{i-j} \le M$. Next suppose that $u$ and $v$ lie in different fibers. Then, since $d_\bund(u,v)=1$, we have $u\in T_h$ and $v\in T_{h'}$ for some $h,h'\in \Gamma$ with $d_\base(h,h') = 1$. Let $i,j\in \Z$ be such that $\fiberwise_a(u)\in Z_i^h$ and $\fiberwise_a(v)\in Z_j^{h'}$, and note that $d_\bund(\fiberwise_a(u),\fiberwise_a(v))\le \fiberlip$ by \Cref{lem:fiberwise_lipschitz}. Since $\fiberwise_a(u)\in [\Sigma_i(h),\Sigma_{i+1}(h))$, it follows that $\fiberwise_a(u)$ lies between $\Sigma_{i-1}(h)$ and $\Sigma_{i+2}(h)$ with $d_h(\fiberwise_a(u),\Sigma_n(h)) \ge \separation$ for $n \in\{i-1,i+2\}$. Therefore we may apply \Cref{lem:separated_sections} to conclude that $\fiberwise_a(v)$ lies between $\Sigma_{i-1}(h')$ and $\Sigma_{i+2}(h')$ in $\axis_{h'}(a)$. In particular, we must have $j\in \{i-1,i,i+1,i+2\}$ so that $d_g(\Pi_a(u),\Pi_a(v)) = d_g(\Sigma_i(g),\Sigma_j(g)) = M\abs{i-j} \le 3M$. This completes the proof of the claim.

We now prove the proposition. Let $x,y\in \axis_g(a)$ be arbitrary. Then clearly $d_\bund(x,y)\le d_g(x,y)$ by definition of the path metrics $d_\bund$ and $d_g$. Choosing vertices $x',y'\in \axis_a(g)$ with $d_g(x,x'),d_g(y,y')\le 1$, the claim and triangle inequality together imply that 
\[d_g(x,y) \le 2 + 6M + d_g(\Pi_a(x'),\Pi_a(y')) \le 2 + 6M + 3M d_\bund(x',y') \le 2 + 6M + 3M(d_\bund(x,y) + 2).\] 
Therefore the inclusion $(\axis_g(a),d_g)\to (\bund,d_\bund)$ is a $(12M+2)$--quasi-isometric embedding.
\end{proof}

\subsection{A Scott--Swarup theorem} \label{sec:ScSw}
In \cite{scott1990geometric}, Scott and Swarup proved that a finitely generated, infinite index subgroup of the fiber of a fibered hyperbolic $3$-manifold group is quasiconvex. This result was extended to arbitrary hyperbolic extensions of surface groups in \cite{dowdall2014pseudo} and to hyperbolic free-by-cyclic groups with fully irreducible monodromy in \cite{mitra1999theorem}. In this section, we generalize these results on the nondistortion of finitely generated, infinite index subgroups of fiber group to the case of hyperbolic extensions of free group by convex cocompact subgroups of $\Out(\free)$.  

We first show each free factor $A$ of $\free$ is undistorted in $E_\Gamma$. Our proof uses the following well-known fact about hyperbolic groups:

\begin{fact}\label{claim:hyperbolic_endpoints}
Suppose that $G$ is a hyperbolic group and let $a,b \in G$ be infinite order elements. Then $(a^nb^{-n})^\infty \to a^\infty$ and $(a^nb^{-n})^{-\infty} \to b^\infty$ in $\partial G$ as $n \to \infty$.
\end{fact}

\begin{proposition} \label{prop:factor_distortion}
Let $\Gamma \le \Out(\free)$ be a finitely generated group with quasi-isometric orbit map into $\CS$ and let $E_\Gamma$ be the associated hyperbolic extensions of $\free$. Then for any free factor $A$ of $\free$, $A$ is quasiconvex in $E_\Gamma$.
\end{proposition}

\begin{proof}
Let $\delta$ be the hyperbolicity constant of $E_\Gamma$.
By \Cref{lem:near_fiber} and \Cref{width}, there are constants $R,Q>1$ so that $\diam_\Gamma(p(x^*)) \le R$ for each simple element $x$ of $\free$ and that $x^*$ and the $Q$--quasigeodesic $\axis_g(x)$ have Hausdorff distance at most $R$ whenever $x$ is minimized in the fiber over $g$, i.e. whenever $g\in \minset{\Gamma}{\rose}([x])$.

Now let $a,b \in A$ be arbitrary. We claim that $d_\Gamma(g,h)\le 5R+4\delta$ for any $g\in \minset{\Gamma}{\rose}([a])$ and $h\in \minset{\Gamma}{\rose}([b])$.
Since $a^nb^{-n} \in A$, these elements are simple. Moreover, since $(a^nb^{-n})^\infty \to a^\infty$ and $(a^nb^{-n})^{-\infty} \to b^\infty$ in $\partial E_\Gamma$ by \Cref{claim:hyperbolic_endpoints}, there is an $N\ge0$ such that $(a^Nb^{-N})^*$ meets a $2\delta$--neighborhood of $a^*$ and $b^*$ in $E_\Gamma$. Then
\begin{align} \label{gets_close}
\diam (p(a^*) \cup p(b^*))  \le \diam p(a^*) + 2\delta+ \diam p((a^Nb^{-N})^*) +2\delta +  \diam p(b^*)  \le 3R+4\delta.
\end{align}
Therefore $d_\Gamma(g,h)\le 5R+4\delta$ as claimed. Setting $D = d_\Gamma(1,g)+5R+4\delta$, this moreover shows that $d_\Gamma(1,h)\le D$ whenever $h$ lies in the minimizing set $\minset{\Gamma}{\rose}([b])$ for any $b\in A$.

We can now directly verify that $A$ is quasiconvex in $E_\Gamma$. Identify $A$ with the vertices of the tree $T_1^A$ in $\bund$. For any two vertices $a,b$ of $T_1^A$ there is an $x \in A$ whose axis $\mathcal{A}_{1}(x)$ in $T_1$ passes through the vertices $a$ and $b$. If $x$ is minimized in the fiber over $h \in \Gamma$, then  $d_\Gamma(1,h) \le D$. By \Cref{eqn:bilipschitz_neighborhs}, there exists an $\free$--equivariant $\bblip^D$--bilipschitz map $T_1\to T_h$ (obtained by writing $h$ as a geodesic $s_1\dotsb s_n$ in $\Gamma$ and lifting the $s_i\in S$ to generators $t_i\in \tilde{S}$). This and the fact that $\axis_h(x)$ is a $Q$--quasigeodesic together imply that $\axis_1(x)$ is a $Q'$--quasigeodesic in $\bund$ for some constant $Q'\ge 1$ depending only on $Q$, $\bblip$, and $D$.
Hence, by \Cref{prop:stability_of_quasis}, any geodesic joining $a$ and $b$ in $E_\Gamma = \bund^0$ stays within the $R_0(Q',\delta)$--neighborhood of the geodesic $\axis_1(x)$ in $T^A_1$ joining $a$ and $b$. Therefore $A$ is quasiconvex in $E_\Gamma$.
\end{proof}

We can now combine \Cref{prop:factor_distortion} with \Cref{prop:lifting_qi} to prove the main result of this section. 

\begin{theorem}[Nondistortion in fibers]
 \label{th:main_ss}\label{th:scott_swarup}
Suppose that $\Gamma\le\Out(\free)$ quasi-isometrically embeds into $\cs$, and let $L$ be a finitely generated subgroup of the fiber $\free\lhd E_\Gamma$. Then $L$ is quasiconvex, and hence undistorted, in the hyperbolic extension $E_\Gamma$ if and only if $L$ has infinite index in $\free$.
\end{theorem}

\begin{proof}
Suppose that $L$ is a finitely generated, infinite index subgroup of $\free$. By Marshall Hall's theorem, $L$ is a free factor of $H$ for some finite index subgroup $H \le \free$. By 
\Cref{prop:lifting_qi}, the group $\Gamma^H$ qi-embeds into $\CS(H)$, and hence the corresponding $H$--extension $E_{\Gamma^H}$ fitting into the sequence
\[
1 \longrightarrow H \longrightarrow E_{\Gamma^H} \longrightarrow \Gamma^H \longrightarrow 1
\]
is hyperbolic by \Cref{th:DT1} and \Cref{th:qi_into}. Since $L$ is a free factor of $H$, \Cref{prop:factor_distortion} implies that $L$ is quasiconvex in $E_{\Gamma^H}$. Finally, since $H$ has finite index in $\free$, $E_\Gamma$ and $E_{\Gamma^H}$ are commensurable, and we conclude that $L$ is quasiconvex in $E_\Gamma$. Conversely, if $L$ has finite index in $\free$, then $L$ is quasi-isometric to $\free\lhd E_\Gamma$ which itself is exponentially distorted in $E_\Gamma$ by virtue of being infinite and normal.
\end{proof}

\begin{remark}
We note that the above theorem does not necessarily hold for hyperbolic extensions of $\free$ by groups that do not admit quasi-isometric orbit maps into $\CS$. For example, if $\phi$ is an automorphism of $\free$ which is atoroidal but fixes the conjugacy class of a free factor $A$, then the $\free$--extension $\free \rtimes \langle \phi \rangle$ is hyperbolic by \cite{Brink}, but the subgroup $A$ is not quasiconvex.
\end{remark}

\section{Hyperbolicity of $E_\Gamma$ and convex cocompactness of $\Gamma$}
\label{sec:width}
In the previous section, we learned that if $\Gamma$ is convex cocompact and purely atoroidal then not only is the extension $E_\Gamma$ hyperbolic, but the projection $E_\Gamma \to \Gamma$ has controlled geometry over the axes of simple elements. In this section, we develop a converse to the main theorem of \cite{DT1}, which established hyperbolicity of $E_\Gamma$. That is, we impose additional structural properties of $E_\Gamma$ that imply the induced orbit map $\Gamma \to \F$ is a quasi-isometric embedding. These properties turn out to characterize convex cocompact subgroups of $\Out(\free)$ among the class of subgroups inducing hyperbolic extensions of $\free$.

Suppose henceforth that $1\to \free\to E\stackrel{p}{\to} Q\to 1$ is a hyperbolic extension of $\free$. This short exact sequence induces an outer action of $Q$ on $\free$ given by the homomorphism $Q\to \Out(\free)$ sending $q\in Q$ to the class of the automorphism that conjugates $\free\lhd E$ by any lift $\tilde{q}\in E$ of $q$. We then have the commutative diagram
\begin{equation}\label{eq:extension_diagram}
		\begin{tikzpicture}[>= to, line width = .075em, 
			baseline=(current bounding box.center)]
		\matrix (m) [matrix of math nodes, column sep=1.5em, row sep = 1.5em, 		text height=1.5ex, text depth=0.25ex]
		{
			1 & \free  & E & Q & 1 \\
			1 & \free  & E_\Gamma & \Gamma & 1,\\
		};
		\path[->,font=\scriptsize]
		(m-1-1) edge 					(m-1-2)
		(m-1-2) edge 					(m-1-3)
		(m-1-3) edge node[above] {$p$}	        	(m-1-4)
		(m-1-4) edge 					(m-1-5)
		
		(m-2-1) edge 					(m-2-2)
		(m-2-2) edge    	                        (m-2-3)
		(m-2-3) edge 	                                (m-2-4)
		(m-2-4) edge 					(m-2-5)
		
		(m-1-3) edge					(m-2-3)
		(m-1-4) edge (m-2-4)
		;
		\draw[double distance = .15em,font=\scriptsize]
		(m-2-2) -- 					(m-1-2)
		;
		\end{tikzpicture}
\end{equation}
where  $\Gamma$ is the image of $Q \to \Out(\free)$.
Fixing finite generating sets for $E$ and $Q$, for each element $a\in \free$ we continue to write $a^*$ for a geodesic in $E$ joining $a^{-\infty}\in\partial E$ to $a^{\infty}\in \partial E$. The image $p(a^*)$ in $Q$ then depends only on the $\free$--conjugacy class $\alpha$ of $a$. Hence, as in \S\ref{sec:width_theorem}, we may define the \define{width} of $a\in \free$ (or $\alpha$) to be
\[
\mathrm{width}_Q(a) = \mathrm{width}_Q(\alpha) \colonequals \diam_Q p(\alpha^*).
\]

\begin{theorem}[Convex cocompactness]
\label{th:intro_width}\label{th:cc}
Suppose that $1 \to \free \to E \to Q \to 1$ is a hyperbolic extension of $\free$. Then $Q$ has convex cocompact image in $\Out(\free)$ (and hence admits a quasi-isometric embedding orbit map into $\cs$) if and only there exists $D\ge 0$ so that $\width_Q(a)\le D$ for each simple element $a\in \free$.
\end{theorem}
\begin{proof}
Since $E$ is hyperbolic, the induced homomorphism $\chi \colon Q\to \Out(\free)$ must have finite kernel, otherwise $E$ contains a subgroup isomorphic to $\free \times \mathrm{ker}(\chi)$  (see \cite[\S 2.5]{DT1}). 
Thus $Q$ is quasi-isometric to its image $\Gamma$ in $\Out(\free)$ and, further, each vertical arrow in  \eqref{eq:extension_diagram} has finite kernel.
From this we see that $E$ is $\free$--equivariantly quasi-isometric to $E_\Gamma$ and moreover that $\mathrm{width}_Q(a)$ coarsely agrees with the width $\width(a)$ in $\Gamma$ as defined in \S\ref{sec:width_theorem}. Therefore if $\Gamma$ is convex cocompact, \Cref{width} shows that supremum $\sup_a\mathrm{width}_Q(a)$ over all simple $a\in \free$ is bounded. 

For the converse, suppose that $\mathrm{width}_Q(\beta) \le D$ for each simple conjugacy class $\beta$ of $\free$. Since $E_\Gamma$ is $\delta$--hyperbolic, $\Gamma$ is purely atoroidal and so it suffices to show that $\Gamma$, or equivalently $Q$, qi-embeds into $\F$ by \Cref{th:qi_into}. As it is more natural for our argument, we instead work with the quasi-isometric primitive loop graph $\pl$ defined in \S\ref{sec:co-surface graph}. 
Fix $\alpha \in \pl^0$ and consider the orbit map $Q \to \pl$ given by $g \mapsto g\cdot\alpha$, where $Q$ acts on $\pl$ via $Q\to \Out(\free)$. We define a coarse map $\sigma\colon \pl \to Q$ which we show is a coarse Lipschitz retraction for the orbit map $Q \to \pl$. Since the orbit map is necessarily Lipschitz, this will show that $Q \to \pl$ is a quasi-isometric embedding and establish the theorem.
Set 
\[
\sigma(\beta) = p(\beta^*),
\]
which is by assumption a subset of $Q$ of diameter at most $D$.  This map is equivariant since for each $g\in Q$ and any lift $\tilde{g} \in E$,
\[
\sigma(g\cdot\beta) = \sigma(\tilde{g}b) = p(\tilde g b^*) = gp(\beta^*),
\]
where $b$ is any representative of the conjugacy class of $\beta$. Hence, if we set $D_0 = \diam_Q(\{1\}\cup \sigma(\alpha))$, then $\sigma \colon \pl \to \Gamma$ is indeed a $D_0$--coarse retraction and so it only remains to show that it is Lipschitz.

Let $\beta$ and $\gamma$ be adjacent conjugacy classes in $\pl$ and choose representatives $b$ and $c$ such that $\langle b,c \rangle$ is a rank $2$ free factor of $\free$. Then, for each $n\in \Z$, $b^nc^n$ is a simple element of $\free$, and by \Cref{claim:hyperbolic_endpoints} $(b^nc^n)^{\infty}$ approaches $b^\infty$ as $n\to \infty$ and $(b^nc^n)^{-\infty}$ approaches $c^{-\infty}$ as $n \to \infty$. Hence, the axis $(b^nc^n)^*$ in $E$ becomes forward asymptotic to $b^*$ and backward asymptotic to $c^*$. Then, just as in \Cref{gets_close}, for all sufficiently large $n$ we have that
\[
\diam_Q (p(b^*) \cup p(c^*)) \le \diam_Q p(b^*) + 2\delta +  \diam_Q p((b^nc^n)^*) + 2\delta + \diam_Q p(c^*)\le 3D + 4\delta.
\]
This demonstrates that $\sigma\colon \pl \to Q$ is a Lipschitz retract and completes the proof.
\end{proof}

\bibliographystyle{alphanum}
\bibliography{geometric_properties}

\bigskip

\noindent
\begin{minipage}{.55\linewidth}
Department of Mathematics\\
Vanderbilt University\\
1326 Stevenson Center\\
Nashville, TN 37240 , U.S.A\\
E-mail: {\tt spencer.dowdall@vanderbilt.edu}
\end{minipage}
\begin{minipage}{.45\linewidth}
Department of Mathematics\\ 
Yale University\\ 
10 Hillhouse Ave\\ 
New Haven, CT 06520, U.S.A\\
E-mail: {\tt s.taylor@yale.edu}
\end{minipage}

\end{document}